\documentclass[12pt]{amsart}    
\usepackage{amssymb,amsmath,eepic,verbatim, bm}
\usepackage{pstricks,pst-node}

\psset{unit=1pt}
\psset{arrowsize=4pt 1}
\psset{linewidth=.5pt}

\countdef\x=23
\countdef\y=24
\countdef\z=25
\countdef\t=26

\def\graybox(#1,#2){
\x=#1 \y=#2 
\z=\x \t=\y
\advance\z by 10 
\advance\t by 10 
\psframe[fillstyle=solid,fillcolor=lightgray,linewidth=0pt](\x,\y)(\z,\t) 
\psline[linewidth=.5pt](\x,\y)(\x,\t)(\z,\t)(\z,\y)(\x,\y)}

\def\emptygraybox(#1,#2){
\x=#1 \y=#2 
\z=\x \t=\y
\advance\z by 10 
\advance\t by 10 
\psframe[fillstyle=solid,fillcolor=lightgray,linewidth=0pt,linecolor=lightgray](\x,\y)(\z,\t)}

\def\blankbox(#1,#2){
\x=#1 \y=#2 
\z=\x \t=\y
\advance\z by 10 
\advance\t by 10 
\psframe[linewidth=0pt](\x,\y)(\z,\t)}

\renewcommand{\P}{\mathbb{P}}
\newcommand{\R}{\mathbb{R}}
\newcommand{\Z}{\mathbb{Z}}
\newcommand{\Q}{\mathbb{Q}}

\renewcommand{\setminus}{\smallsetminus}
\renewcommand{\phi}{\varphi}

\renewcommand{\tilde}{\widetilde}
\renewcommand{\hat}{\widehat}
\renewcommand{\bar}{\overline}

\newcommand{\on}{\operatorname}

\newcommand{\bull}{ {\scriptscriptstyle{\bullet}}  }

\newcommand{\bk}{\mathbf{k}}
\newcommand{\bp}{\mathbf{p}}
\newcommand{\bq}{\mathbf{q}}

\DeclareMathOperator{\rk}{rk}

\newcommand{\triple}{{\bm{\tau}}}

\newcommand{\Pf}{ \on{Pf} }

\newcommand{\rank}{ \on{rank} }
\newcommand{\Hom}{ \on{Hom} }
\newcommand{\Ker}{ \on{Ker} }

\def\cprime{\char126}   

\newtheorem{thm}{Theorem}

\newtheorem{theorem}{Theorem}[section]
\newtheorem{lemma}[theorem]{Lemma}
\newtheorem{proposition}[theorem]{Proposition}

\newtheorem*{thm*}{Theorem}
\newtheorem*{lem*}{Lemma}
\newtheorem*{prop*}{Proposition}
\newtheorem*{cor*}{Corollary}

\theoremstyle{definition}

\newtheorem*{defn*}{Definition}
\newtheorem*{rmk*}{Remark}

\begin{document}

\title[Degeneracy Loci, Pfaffians, Vexillarity]{Degeneracy Loci, Pfaffians, and Vexillary Signed Permutations in Types B, C, and D}
\date{October 5, 2012}

\author{D. Anderson}
\address{FSMP--Institut de Math\'ematiques de Jussieu, 75013
Paris, France}
\email{andersond@math.jussieu.fr}

\author{W. Fulton}
\address{Department of Mathematics,
University of Michigan,
Ann Arbor, Michigan  48109-1043, U.S.A.}
\email{wfulton@umich.edu}

\begin{abstract}
We define a notion of \emph{vexillary signed permutation} in 
types B, C, and D, corresponding to natural degeneracy loci 
for vector bundles with symmetries of those types.  We show that the classes of these loci are given by explicit Pfaffian formulas.  The Grassmannian formulas of Kazarian are important special cases, and the corresponding 
double Schubert polynomials of Ikeda, Mihalcea, and Naruse
are shown to be equal to these Pfaffians.
\end{abstract}

\maketitle

%

\section*{Introduction} 

The Giambelli-Thom-Porteous formula gives a formula for the cohomology 
(or Chow) class of a locus where two general subbundles  $E$  and  $F$  of a 
given vector bundle  $V$  meet in at least a specified dimension.  Kempf and 
Laksov \cite{KL} generalized this to give a determinantal formula for the locus 
where a given subbundle  $E$  meets each member of a flag of bundles 
$F_{\bull} = F_1 \subset F_2 \subset \ldots \subset V$; their polynomials 
have since reappeared under various names like multi-Schur or generalized 
factorial Schur polynomials.  Lascoux and Sch{\"u}tzenberger \cite{LS} defined 
double Schubert polynomials $\frak{S}_w(x,y)$  for every permutation  $w$, 
which have the stability property that the polynomial does not change 
when a permutation is regarded naturally in a larger symmetric group.
These polynomials give formulas for loci where two general flags $E_{\bull}$ and 
$F_{\bull}$ meet in given dimensions (\cite{F2}).  Lascoux and Sch{\"u}tzenberger also 
defined a notion of \emph{vexillary} permutation, whose Schubert polynomials 
are given by generalized Schur determinants, or \emph{multi-Schur determinants}: 
\[
s_{\lambda}(c(1), \ldots, c(r)) = |c(i)_{\lambda_i+j-i}|.
\]
Here $c(i) = 1 + c(i)_1 + c(i)_2 + \dots$, 
with the $c(i)_j$ commutating variables.  In \cite{F2} these loci were shown to  
correspond to degeneracy loci with particularly simple descriptions, which are 
further simplified in Section 1 of this paper.

Pragacz (\cite{P1}, \cite{P2}) showed that Schubert classes in Grassmannians of 
isotropic subspaces of a vector space with a nondegenerate 
skew-symmetric or symmetric form can be given by polynomials 
known as Schur Q and P-functions, which are certain universal 
polynomials that can be written as Pfaffians. Billey and Haiman \cite{BH} gave 
polynomials (actually power series) for Schubert varieties in flag varieties of types B, 
C, and D.   The rough state of knowledge in 1995, 
including several attempts at general formulas, can 
be found in \cite{FP}; see \cite{KT},  \cite{BKT1}, \cite{BKT2} and \cite{PR} for progress 
since then.

In a remarkable preprint \cite{K} a decade ago, Kazarian gave  
Pfaffian formulas for those degeneracy loci that describe how one 
isotropic subbundle in a bundle with a skew or symmetric form 
meets a general flag of isotropic subbundles.  His formulas are generalized 
Pfaffians, which it is natural to call \emph{multi-Schur Pfaffians}, and denote 
by $\Pf_{\lambda}(c(1), \dots, c(r))$.  Here $\lambda$ is a strict partition of 
length $r$, i.e. $\lambda_1 > \lambda_2 > \dots > \lambda_r > 0$, and $c(i) = 
1 + c(i)_1 + c(i)_2 + \dots$, with commuting variables as before; in addition, 
one requires that the matrix with $(k,l)$-entry 
\[
c(k)_{\lambda_k} c(l)_{\lambda_l} + 2 \sum_{j=1}^{\lambda_l} (-1)^j
 c(k)_{\lambda_k + j} c(l)_{\lambda_l - j}
 \] 
 be skew-symmetric.  When $r$ is even, $\Pf_{\lambda}(c(1), \dots, c(r))$ is the 
 Pfaffian of this matrix.  When $r$ is odd, 
 \[\Pf_{\lambda}(c(1), \dots, c(r)) = \sum_{k=1}^r (-1)^{k-1} c(k)_{\lambda_k} 
 \Pf_{\lambda_1  \ldots \hat{\lambda_k} \ldots \lambda_r}(c(1), \dots  \hat{c(k)} \dots, c(r)).
 \]
Alternatively, one can set $\lambda_{r+1} = 0$ and use the formula in the even case.  
In particular, $\Pf_{\lambda_1}(c(1)) = c(1)_{\lambda_1}$, and 
a Pfaffian for the empty partition is $1$.  (Basic results about Pfaffians can 
be found in \cite{Knuth}, \cite{K}, and Appendix D of \cite{FP}.)
 
More recently, Ikeda, Mihalcea, and Naruse (\cite{I}, \cite{IN}, \cite{IMN}) have 
succeeded in defining general double Schubert polynomials for 
signed permutations of types B, C, and D, which have the same 
stability property as in type A. Their work is 
in the context of equivariant cohomology, but this is very close 
to the general degeneracy loci setting.  Their work grows from that of 
Billey and Haiman, and the corresponding 
single polynomials (for ordinary cohomology) are exactly those of \cite{BH}.  
They show that those corresponding to the longest elements in any of the finite 
Weyl groups are Pfaffians, and they show that some of Kazarian's 
formulas (for Lagrangian subbundles in type C) are special cases of 
theirs.  Modified forms of multivariable versions of the Schur P- and Q-functions, 
which had been developed by Ivanov \cite{Iv}, play an important role in their work.

In this paper we define a notion of vexillary signed permutation 
in types B, C, and D.  We show that they correspond to degeneracy loci 
of a particularly simple form, and we show that the double Schubert 
polynomials for 
their classes are given by explicit Pfaffians.  These include all the 
Pfaffians in \cite{IMN} and \cite{K} as special cases.
  
Billey and T.~K.~Lam \cite{BL} defined a notion of vexillary in types B, 
C, and D, by requiring the corresponding single Schubert polynomials of 
\cite{BH} to be equal to the corresponding Stanley symmetric 
polynomial.  Our vexillary elements form subsets of theirs, 
which coincide in type B. 
Billey and Lam ask if there is a class whose 
polynomials are Pfaffians, which is what we answer here, for double 
as well as single polynomials.

\smallskip

We describe our main result in type C, which requires 
reviewing some language from \cite{IMN}; details will appear in Section 2. 
A key role is played by the 
ring $\Gamma$, which is a quotient of a polynomial ring in 
indeterminates $Q_k$, with one generator and relation for each positive integer $k$:
\[
 \Gamma = \Z[Q_1, Q_2, \dots]/(Q_k^2 + 2\sum_{j=1}^k (-1)^j Q_{k+j}Q_{k-j}, \,\, k = 1, 2, \dots) .
\]
We denote by the same $Q_k$ the image of $Q_k$ in $\Gamma$, with $Q_0 = 1$.  These 
relations say that $Q \cdot Q^* = 1$, where $Q$ is the power series $1 + Q_1 + 
Q_2 + \dots$, and $Q^* = 1 - Q_1 + Q_2 - \dots$.  We can therefore 
substitute for $Q_i$ the corresponding term $A_i$ of any power series $A = 
1 + A_1 + A_2 + \dots$ that satisfies the identity $A \cdot A^* = 1$.
\footnote{If one takes $A = \prod_{i=1}^{\infty} (1+x_i)/(1-x_i)$, as done in \cite{IMN}, 
one gets an embedding of $\Gamma$ in the ring of symmetric functions in the $x$ variables, 
but we need specializations that are not of this form.  And we use $x$ (and $y$) variables for 
Chern classes, to maintain the analogy with Schubert polynomials of type A.}

For $k \neq l$,  set
\[
Q_{k \, l} = Q_kQ_l +  2\sum_{j=1}^l (-1)^j Q_{k+j}Q_{l-j} .
\] 
The relations say that $Q_{k \, k} = 0$ and $Q_{k \, l} = -Q_{l \, k}$.  Hence, for any strict 
partition $\lambda = \lambda_1 > \dots > \lambda_r > 0$, the matrix with 
$(i,j)$-entry $Q_{\lambda_i \, \lambda_j}$ is skew symmetric, so we can form 
its Pfaffian, which is denoted by $Q_\lambda$; equivalently, 
$Q_\lambda = \Pf_\lambda(Q, \dots, Q)$.  
These $Q_\lambda$, as  $\lambda$ varies over all strict partitions, 
form a basis for $\Gamma$ over $\Z$.  

For any indeterminates $t_1, t_2, \dots $, 
we write $\Gamma[t]$ for   $\Gamma[t_1, t_2, \dots ]$, and call elements of 
such a ring \emph{polynomials}, a slight abuse of language.  We require the 
\emph{multivariate} version of these $Q_\lambda$
polynomials, due to Ivanov [Iv],
which are elements in the ring $\Gamma[t]$; we denote them by $\Q_\lambda$ 
or $\Q_\lambda(t)$, 
a departure from the notation of \cite{Iv} or \cite{IMN}.  These can also be 
defined as Pfaffians, which is the approach we take, completing somewhat the 
treatment in \cite{Iv} and \cite{IMN}.  First
define $q(k) = Q \cdot \prod_{j=1}^{k-1}(1 + t_j)$, regarding $Q = 1 + Q_1 + \dots$ 
as a power series, so $q(k)_p = Q_p + e_1 Q_{p-1} + e_2 Q_{p-2} + \ldots$, where  
$e_i = e_i(t_1, \dots, t_{k-1})$ is the $i^{\text th}$ elementary symmetric 
polynomial in these variables..  
Set   
\[
\Q_k = q(k)_k  = Q_k  + e_1 Q_{k-1} + e_2 Q_{k-2} + \dots + e_{k-1}Q_1,
\]
Then, for nonnegative integers $k$ and $l$, define 
\[
\Q_{k\,l} = q(k)_k q(l)_l + 2 \sum_{j=1}^l (-1)^j q(k)_{k+j} q(l)_{l-j} ,
\]
One verifies that  $\Q_{k\,k} = 0$ and $ \Q_{l\,k} = -\Q_{k\,l} $, so we can
define the general $\Q_\lambda$ to be the Pfaffian of the skew-symmetric 
matrix whose $(i,j)$-entry is $\Q_{\lambda_i \, \lambda_j}$.  Equivalently, 
 $\Q_\lambda = \Pf_{\lambda}(c(1), \dots, c(r))$, where $c(k) = q(\lambda_k) = 
 Q \cdot \prod_{j=1}^{\lambda_k-1}(1 + t_j)$.
 \footnote{
Our polynomial $\Q_\lambda$ is what is denoted by $Q_\lambda(x | -t)$ in 
\cite{Iv} and \cite{IMN}, and our $Q_\lambda$ is their $Q_\lambda(x)$.  
We have eliminated the $x$ variables, which they use to embed 
$\Gamma$ in a ring of symmetric functions -- for which we have no use --
 and we have changed the signs of the $t$ variables, to avoid having to keep track of 
 unnecessary signs.  These formulas are more explicit than those given in  \cite{IMN}, 
see Proposition 2.6 below.}

\smallskip

The polynomials we will construct for vexillary signed permutations in type C, and all the 
double Schubert polynomials of type C,  are elements 
in the polynomial ring $\Gamma[x,y] = \Gamma[x_1, x_2, \dots, y_1, y_2, \dots ]$ 
in two sets of variables.  In the vexillary case, they can be obtained by substituting some of the 
$x$'s and $y$'s for some of the $t$'s in these multi-Schur Pfaffians.
We define these vexillary signed permutations and their polynomials 
from the notion of a {\bf triple of type C}, by which we mean 
a triple $\bm{\tau} = (\bf{k}, \bf{p}, \bf{q})$ 
of sequences of positive integers of the same length  $s$, 
\[
{\bf{k}} \, = \,   k_1 < \dots < k_s, \,\,\,\,\,\, {\bf{p}} \, = \,  p_1 \geq \dots \geq p_s , \,\,\,\,\,\,
{\bf{q}} \, = \, q_1 \geq \dots \geq  q_s ,
\]
satisfying the conditions that  
\[
\tag{*} \,\,\,\,\, (p_i - p_{i+1}) + (q_i - q_{i+1}) \, > \, k_{i+1} - k_i 
\]
for $1 \leq i \leq s-1$.  To each triple one can write down a signed 
permutation $w(\bm{\tau})$, which can be defined as the $w$ of minimal length such 
that, for each $i$ from $1$ to $s$, the number of $a$ with $a \geq p_i$ 
such that  $w(a)$ is a barred integer $\bar{b}$ with $b \geq q_i$, is equal to 
$k_i$.  These are our vexillary signed permutations of type C.  
Such a triple defines a strict partition $\lambda = \lambda({\bm{\tau}})$; it is 
the strict partition with  $\lambda_{k_i} = p_i + q_i - 1$, and minimal otherwise; that is, 
\[
 \lambda_k = \lambda_{k_i} + k_i - k ,
\] 
where  $k_i$ is the minimal integer in the sequence $\bf{k}$ with $k \leq k_i$. 

We define a polynomial ${\bf{Q}}_{\bm{\tau}}(x,y)$ in $\Gamma[x,y]$ by substituting 
$x_1, \dots, x_{p_s - 1}$ and $y_1, \dots, y_{q_s - 1}$ (in any order) for the first 
$\lambda_{k_s} - 1$ values of $t_1, t_2, \dots $ in $\Q_{\lambda}(t)$;  then substitute 
the next  $p_{s-1} - p_s$ $x_i$'s and the next $q_{s-1} - q_s$ $y_i$'s, in any order, 
and continue until the values $x_1, \dots, x_{p_1 - 1}$ and $y_1, \dots, y_{q_1 - 1}$
have been substituted.  Equivalently, set
\[
 c(k_i) = Q \cdot \prod_{j=1}^{p_i - 1} (1 + x_j) \cdot \prod_{j=1}^{q_i - 1} (1 + y_j), 
\]
and set $c(k) = c(k_i)$, where $k_i$ is minimal greater than or equal to $k$. 
Then ${\bf{Q}}_{\bm{\tau}}(x,y) = \Pf_{\lambda}(c(1), \dots, c(k_s))$, 
the Pfaffian of the matrix whose $(k,l)$ entry is 
$c(k)_{\lambda_k} c(l)_{\lambda_l} + 2 \sum_{j=1}^{\lambda_l} (-1)^j c(k)_{{\lambda_k}+j} c(l)_{{\lambda_l}-j}$.

These polynomials are used to construct the double Schubert polynomials of type C, 
for any element in the infinite Weyl group $W = \bigcup W_n$, i.e., to any signed permutation $w$, as follows.  
We denote\footnote{The polynomial we denote by  $\frak{C}_w(x,y)$ is corresponds to what would be called $\frak{C}_w (x, -y; z)$ in \cite{IMN} --- 
 although the notation they use for their variables is different, and 
 they write  $\frak{C}_w (z, t; x)$, so their $z_i$ is our $x_i$, 
 their $t_i$ is our $-y_i$, and their $x_i$ is the 
``dummy'' variable that we avoid using.   Similarly for $\frak{B}_w(x,y)$ and 
$\frak{D}_w(x,y)$ in types B and D. Our sign convention 
makes all the monomials appearing in all Schubert polynomials of 
types B, C, and D positive.} 
  these polynomials, which live in $\Gamma[x,y]$, by $\frak{C}_w(x,y)$.  If $\bm{\tau}$ is the 
triple with ${\bf{k}} = 1 \, 2 \, \dots \, n$ and  ${\bf{p}} = {\bf{q}} = n \,\, n\!-\!1 \, \dots \, 1$, then
$w(\bm{\tau})$  is the element $w_\circ^{(n)}$ of longest length in $W_n$.  
In this case $\lambda = (2n\!-\!1, 2n\!-\!3, \dots, 1)$.  The double Schubert polynomial $\frak{C}_{w_\circ^{(n)}}(x,y)$ can 
be defined to be the polynomial ${\bf{Q}}_{\bm{\tau}}(x,y)$.  For any other $w$, choose $n$ so that $w$ is in 
$W_n$, and write $w = w_\circ^{(n)} \cdot s_{i_1} \dots s_{i_{\ell}}$ where the $s_i$ are the standard generators, 
and the length of $w$ is $\ell$ less than the length of $w_\circ^{(n)}$, and then set 
\[
 \frak{C}_w(x,y) = \partial_{i_{\ell}} \circ \dots \circ \partial_{i_1} (\frak{C}_{w_\circ^{(n)}}(x,y)).
\]
The key in \cite{IMN} to making this work is showing that this polynomial is independent of the 
choice of $n$, which is a corollary of our theorem. In fact,  $\Gamma[x,y]$ is canonically isomorphic 
to the graded inverse limit of equivariant cohomology rings of the corresponding flag varieties, as 
the rank $n$ goes to infinity; proving this was a key step in \cite{IMN}; it also follows from our theorem, 
since the Schubert classes form a basis for both rings over $\Z[y]$.

The geometric setting for the 
degeneracy loci is a vector bundle $V$ of some even rank $2n$, 
on a nonsingular variety $X$; $V$ has
a non-degenerate skew form, together with two flags of 
isotropic subbundles $E = E_1 \supset E_2 \supset \dots$ and 
$F = F_1 \supset F_2 \supset \dots$, with $E_i$ and $F_i$ of rank 
$n+1-i$.  In particular,  $E_1$ and $F_1$ are maximal isotropic 
(Lagrangian) subbundles.  The degeneracy locus $\Omega_{\bm{\tau}}$ 
for a triple is defined by the conditions that the dimension of the intersection 
of the fibers $E_{p_i}(x) \cap F_{q_i}(x)$ is at least $k_i$ for each $i$ from $1$ to $s$, 
and every point $x$ of $X$.  We write this condition as 
\[
\dim \, ( \, E_{p_i} \cap F_{q_i} \, ) \, \geq \, k_i, \,\,\,\,\,\, 1 \leq i \leq s.
\]

The non-degenerate form on $V$ makes $V/E$ the dual bundle of $E$, and 
$V/F$ the dual bundle of $F$.  In particular, 
$c(V-E-F) = c(V)/c(E)c(F) = c(E^*)/c(F) = c(F^*)/c(E)$.  
The identities defining $\Gamma$ imply that we may send $Q$ to $c(V-E-F)$, 
i.e., $Q_i$ to $c_i(V-E-F)$.   Set
\[
  x_i = c_1(E_i/E_{i-1}) \,\,\,\,  \text{ and } \,\,\,\, y_i = c_1(F_i/F_{i-1}).
\]

\begin{thm}\label{t.theorem1}
{\rm (1)} The class of $\Omega_{\bm{\tau}}$ is obtained by the substitution of 
$c(V-E-F)$ for $Q$ in ${\bf{Q}}_{\bm{\tau}}(x,y)$. 

\smallskip
\noindent
{\rm (2)} The polynomial ${\bf{Q}}_{\bm{\tau}}(x,y)$ is equal to the double Schubert polynomial 
$\frak{C}_w(x,y)$, where $w = w(\bm{\tau})$.
\end{thm}

The result in (1) can be written as follows:  Set
\[
c(k_i) = c(V-E-F) \cdot c(V/E_{p_i}) \cdot c(V/F_{q_i}),
\] 
and set $c(k) = c(k_i)$, where $k_i$ is minimal greater than or equal to $k$.  Then the class of 
$\Omega_{\bm{\tau}}$ is the Pfaffian $\Pf_\lambda(c(1), \dots , c(r))$, with $r = k_s$.
In fact, one need not have all the bundles present in the flags: it suffices to have each $E_{p_i}$ and $F_{q_i}$ of 
the corresponding codimensions.

\smallskip

When we say that a polynomial $P$ is the class of a degeneracy locus $\Omega$, it should be 
interpreted in the following sense.  When $\Omega$ has the expected codimension (the degree 
of $P$), this is meant literally: for a specified scheme structure on $\Omega$, $[\Omega] = P$ in the 
cohomology or Chow ring of $X$, where the brackets denote the cycle class of a scheme (each 
irreducible component is counted with multiplicity given by the length of its local ring at its generic 
point, as in \cite{F1}).  In particular, this is the case when the bundles are in general 
position.  In general, one can start with $V$ and $F_\bullet$ on $X$, and 
let $\tilde{X} \to X$ be the corresponding flag bundle, with tautological flag $\tilde{E}_\bullet$.  The bundles 
$\tilde{E}_\bullet$ and the pullback of $F_\bullet$  are always in general position. 
The general case is deduced from this. 
In fact, a flag $E_\bullet$ on $X$ corresponds to a section $s \colon X \to \tilde{X}$, 
with $s^{-1}(\tilde{E}_\bullet) = E_\bullet$.  Each 
degeneracy locus $\Omega$ on $X$ is the pullback of the corresponding 
locus $\widetilde{\Omega}$ on $\tilde{X}$. 
Whenever the codimension of $\Omega$ is the expected one, then 
\[
s^*[\widetilde{\Omega}] = [\Omega].
\]
This uses the fact that all these loci, in the corresponding flag bundle, are Cohen-Macaulay. 
These points are discussed in  \cite{F1} \S14, where the refined pullback is used
to construct a class $s^![\widetilde{\Omega}]$ living on $\Omega$, without any hypotheses 
of generality or codimension;  in current language, this can be regarded as the virtual class 
$[\Omega]^{\text{vir}}$ of the locus.   In types A and C, the loci get their scheme structure from 
the (local) vanishing of minors of appropriate ranks.  In type B these may vanish to higher order, 
so $\widetilde{\Omega}$ must be given its reduced structure; and in type D, one must also take 
this locus to be the closure of the locus where all inequalities are equalities.  
The locus $\widetilde{\Omega}$ 
is always reduced and irreducible (provided $X$ is connected, as we assume).  One can also, 
more symmetrically, use flag bundles for both $E_\bullet$ and $F_\bullet$, but this is not necessary.

Note that 
this setting of a flag bundle  $\tilde{X} \to X$ contains the setting for equivariant cohomology, where 
$X$ is (a finite-dimensional approximation to) the classifying space for a torus.  
Although the equivariant setting is a special case of the degeneracy locus setting, the 
Chern classes $y_i$ in that setting are 
algebraically independent, so that any formulas found in the equivariant case are sure to 
be the same in the general degeneracy locus case.

\smallskip

One contribution of this paper is to make the formulas 
of \cite{IMN} more explicit, with direct, geometric proofs.  The bulk of \cite{IMN} depends on localization in equivariant cohomology, together with inductions involving Chevalley and 
transition formulas.  Although localization is available in the 
general setting of degeneracy loci, we do not need it for our main results.
\smallskip

In the definition of a triple, it is useful to allow the strict inequalities of (*) to be replaced by weak inequalities; we call such a triple a {\bf redundant triple}.  Any redundant triple can be reduced to an ordinarary triple by omitting any term $k_i, p_i, q_i$ for which equality holds in (*), and the signed permutation doesn't change by this reduction.  All the above constructions 
and the theorem apply to redundant triples.  For the degeneracy loci, this allows more 
inequalities on dimensions of intersections, although the added ones are in fact redundant.  Note that the terms in matrices for the Pfaffians change, but, in fact, the resulting polynomials are identical.

The case of Theorem 1 when all $p_i$ are equal, so one is measuring how one bundle $E_p$ meets a given flag of some $F_q$'s, is the result of Kazarian \cite{K} in type C.  The loci that are vexillary include all Grassmannian loci for Grassmannians of maximal isotropic subspaces, but not all for Grassmannians of lower dimensional subspaces.  (The nonvexillary loci involve intersections of $E_p$ with duals of some $F_q$'s.)  The maximal elements in each finite Weyl group are vexillary, and the inverse of any vexillary element is vexillary.\footnote{
Our vexillary signed permutations have all their positive values occurring in increasing order, while those for submaximal subspaces would allow two groups of increasing positive values.  The only nonvexillary element of $W_3$ that is not ruled out by having increasing positive values is $w = \bar{3} \, 2 \, \bar{1}$.  Its top term is $Q_{3 \, 2} + Q_{4 \, 1}$, so it cannot be a Pfaffian of the kind we have been considering. 
Of the 48 elements in $W_3$ of type C (or B),  $33$ are vexillary; 19 are 
accounted for by Kazarian (and their duals), with $2$ more in \cite{IMN}.  
Of the $24$ elements in  $W_3$  of type D,  $18$ are vexillary, with 
$9$ from \cite{K}, $2$ more from \cite{IMN}.  As the size of the Weyl group grows, 
the ratio of our new formulas to the previous formulas grows rapidly, 
but, as in type A, the percentage of elements that are not vexillary increases 
even more rapidly.}

\smallskip

The general double Schubert polynomial formula for general degeneracy loci in type C can be deduced from this as follows, as in \cite{BGG} and \cite{D} in the case of flag bundles, and \cite{F2} and \cite{F3} for degeneracy loci.  First complete the flag $F_\bullet$ to a complete flag, by setting $F_{-i} = (F_{i+1})^\perp$ for $i$ positive.  Given a signed permutation $w$ in $W_n$, we have $\Omega_w$, the locus where the dimension of $E_p \cap F_q$ is at least the number of $a \geq p$ such that $-w(a) \geq q$, for $p \in \{1, \dots ,n\}$ and $q \in \{\pm 1, \dots ,\pm n\}$.  In type C, 
all these loci get their scheme structure by the vanishing of the minors that describe the rank conditions.

\begin{cor*}
The class of $\Omega_w$  is the polynomial $\frak{C}_w(x,y)$.
\end{cor*}

A key idea from \cite{K} and \cite{IMN}, to achieve stable polynomials, independent of the ambient 
dimension, is to emphasize codimensions, and avoid the dimension of 
the ambient bundle; this often forces one to write down the right formulas. 
For example, the locus where $E_1$ meets $F_1$ can be given by 
$-x_1 - \dots -x_n - y_1 - \dots -y_n$, 
which does not stabilize as $n$ grows.  Writing this as $c_1(V - E_1 - F_1)$ is the 
stable answer.   

\smallskip

These double Schubert polynomials $\frak{C}_w(x,y)$ satisfy analogues of the main properties of 
those in type A.  In particular, they satisfy the five properties discussed by Fomin and Kirillov \cite{FK}.   
They are (0) homogeneous of degree $\ell(w)$; and (1) related by the corresponding difference 
operators, so represent the corresponding Schubert loci.  They (2) multiply with the same coefficients as 
in the cohomology rings; and (3) have nonnegative coefficients when expressed in terms of monomials.  
They are (4) stable,  
i.e.~independent of the $n$ 
of the Weyl group $W_n$ containing $w$.  

In addition, they are (5) double polynomials satisfying all these 
properties in the equivariant or degeneracy setting, with $\frak{C}_{w^{-1}}(x,y) = \frak{C}_w(y,x)$.  To this, we can add (6) they are given by 
Pfaffian formulas for vexillary signed permutations.  Note that (2) says that 
\begin{equation}
\frak{C}_u(x,y) \, \cdot \, \frak{C}_v(x,y) \, = \, \sum_w  \, c_{u \, v}^w   \,  \frak{C}_w(x,y), 
\end{equation}
for unique polynomials $c_{u \, v}^w$ in $\Z[y]$.  (By Graham's positivity result \cite{G}, these 
coefficients are positive in the variables $y_i - y_{i+1}$ and $-2y_1$, although we are far from 
knowing combinatorial formulas for them.)  The positivity (3) is a nontrivial 
combinatorial result that is proved in \cite{IMN}.  The analogous properties are satisfied 
in types B and D.

This contrasts with the fact proved in \cite{FK} that there are no polynomials satisfying 
properties (0), (1), (2) and (3), even for single polynomials of type $B_2$.  
In that source, as in most of the early work in this area before Billey and Haiman \cite{BH}, 
a polynomial was taken to be 
in $\Z[x]$, or in $\Z[x,y]$ for double polynomials.  Working in the ring $\Gamma[x,y]$, 
so regarding the $Q_{\lambda}$ as monomials, makes this success possible.

\smallskip

The situation in type B is quite similar to that in type C.  We denote the corresponding 
polynomials by $P_\lambda$, $\P_{\lambda}(t)$ and $P_{\bm{\tau}}(x,y)$; as one would 
expect, they come from the $Q_\lambda$, $\Q_{\lambda}(t)$ and $Q_{\bm{\tau}}(x,y)$ 
by division by powers of $2$.  

For 
type D, the corresponding polynomials are considerably more complicated, and we 
denote them by $R_\lambda$, $\R_{\lambda}(t)$ and $R_{\bm{\tau}}(x,y)$ to make the 
distinction clear. 
In addition, in type D, there are two polynomials for each strict partition, depending on 
whether one adds a $0$ to its end or not; so each strict partition occurs in type D with 
both an even and an odd length.  (This corresponds to the two families of maximal 
isotropic subspaces of a vector space with a quadratic form).  
There is one striking relation among them, 
however: when one expresses them appropriately in terms of the basic polynomials, 
the coefficients in all three types agree.  Namely, expand 
\[ 
\P_{\lambda}(t) = \sum p_{\lambda}^\mu(t) \, P_{\mu}, 
\,\,\,\,  \Q_{\lambda}(t) = \sum q_{\lambda}^\mu(t) \, Q_{\mu}, 
\,\,\,\,  \R_{\lambda}(t) = \sum r_{\lambda}^\mu(t) \, P_{\mu},
\]
with coefficients $p_{\lambda}^\mu(t)$, $q_{\lambda}^\mu(t)$, and $r_{\lambda}^\mu(t)$ in $\Z[t]$.  
Then $q_{\lambda}^\mu(t) = p_{\lambda}^\mu(t)$, which is no surprise, but it is also true 
that
\begin{equation}
r_{\lambda}^\mu(t) = p_{\lambda(+)}^{\mu(+)}(t),
\end{equation}
where $\lambda(+)$ is obtained from $\lambda$ by adding $1$ to each of its parts (including 
the last, which may be $0$).  The same relation is true for the corresponding polynomials for 
vexillary elements: writing $Q_{\bm{\tau}}(x,y) = \sum  q_{\bm{\tau}}^\mu(x,y) Q_{\mu}$, with 
coefficients in $\Z[x,y]$, and similarly for types B and D, we have $q_{\bm{\tau}}^\mu(x,y) = 
p_{\bm{\tau}}^\mu(x,y)$, and 
\begin{equation}
r_{\bm{\tau}}^\mu(x,y) = p_{\bm{\tau(+)}}^{\mu(+)}(x,y),
\end{equation}
where $\bm{\tau(+)}$ is obtained from $\bm{\tau}$ by adding $1$ to each $p_i$ and  
$q_i$ of the triple.  It is also the case that, as signed permutations, $w(\bm{\tau(+)}) = w(\bm\tau)$.  
Our proof of this uses specialization results of \cite{IMN}.  It would be 
interesting, and could simplify the type D story significantly, if one could find a direct 
reason for these identities.

\smallskip

We know three ways to prove Theorem 1, and its analogues in types 
B and D.  One 
can assume the results of \cite{IMN}, and then show how to get from 
one vexillary polynomial to another by difference operators -- a 
result of independent interest, which will be carried out in \cite{AF3}; 
this proves (2), and (1) follows.  The general formula can also be 
deduced from Kazarian's result for maximal isotropic bundles, by a 
simple diagonal trick: putting an appropriate skew form on $V \oplus V$, 
the locus where $E_p$ meets $F_q$ in a given dimension is 
the same as the locus where the diagonal $\Delta_V$ meets 
$E_p \oplus F_q$ in the same dimension.  (See Section 2 below.)
The third approach, which we follow here, proves
(1) directly, by finding loci that map birationally to 
the desired loci.  These loci can be defined by vanishing of sections of vector 
bundles; we then apply Kazarian's elegant Gysin formulas to push 
these classes forward.  This latter 
has the advantage that (2) follows, and one recovers the main results of \cite{IMN} 
(and \cite{K}) directly. 

\smallskip
We begin by reviewing type A, in a simpler language than in \cite{F2} which provides a model 
for the other types.  There are 
differences, however, between type A and the others.  In type A the 
flags $E_\bullet$ and $F_\bullet$ play dual roles, one corresponding to 
codimension and the other to dimension, while, in types B, C, and D the 
two flags play interchangeable roles.  (This explains why we have kept 
the sign difference in type A, which is in the original double Schubert 
polynomial notation, but have eliminated it in the other types.)  In types 
B, C, and D there are maximal isotropic subbundles, which is the geometry 
behind the rings $\Gamma$ and $\Gamma'$, while type A has only the 
ring $\Z$. 

Then we 
turn to type C, where we complete the proof of the theorem, and give 
applications to symmetric maps of vector bundles.  Types B and D 
are carried out in the next two sections.  The polynomials in type B 
differ from those in type C by division of some powers of $2$, although there 
is some change in the geometry.  In type D the formulas must be modified more 
significantly, and 
again Kazarian \cite{K} indicates how this modification should be carried out.  As 
the proof in type D was not included in \cite{K}, we include a proof here; this involves 
some algebra, which is put in an appendix.

\smallskip
We have written this paper so that it can be read for those interested only in complex 
varieties, with values in cohomology rings, or for varieties over arbitrary fields, with 
values in Chow rings.  To include characteristic $2$, one needs to make the usual 
modification in types B and D, replacing a symmetric bilinear form $\langle u, v \rangle$ 
by a quadratic norm $N(v)$.  
The interested reader can find a detailed study of this modification in \cite{EKM}, and we 
will not mention it again. 
\smallskip

In \cite{AF2} we show that the vexillary elements can be characterized by 
pattern avoiding, as Lascoux and Sch{\"u}tzenberger did in type A.   
We also show that there is a natural 
notion of essential sets for signed permutations, corresponding to 
which of the conditions defining the locus are necessary; although not obvious from 
this geometric description, this notion turns out to be precisely the notion 
defined by Reiner, Woo and Yong \cite{RWY} for general Coxeter groups.  
In \cite{AF3} we 
give another characterization of vexillary permutations in terms of 
labeled (shifted) Young diagrams, which makes it clear how to go from one 
vexillary class to another by a sequence of difference operators.  In \cite{AF4} 
we show how to extend this story to allow skew symmetric and symmetric 
forms on a vector bundle with values in a line bundle.

\smallskip

All of this work is inspired by \cite{K} and \cite{IMN}, and we are deeply 
indebted to their authors.

\section{Type A revisited; Bundles without Symmetries}

In this section we review the results in type A, from \cite{F2}, but 
modified -- and simplified -- to generalize to the other types.  We are given a 
vector bundle  $V$ of some rank $n$, on a variety $X$,
assumed nonsingular for simplicity, together with two flags of subbundles: 
\[
 V = E_0 \supset E_1 \supset E_2 \supset \ldots  \supset E_a
\]
and  
\[
 0 = F_0 \subset F_1 \subset F_2 \subset \ldots \subset F_b ,
\] 
with the corank of $E_i$ being $i$, and the rank of $F_j$ being $j$. 
For each permutation $w$ in $S_n$, we have the locus $\Omega_w$
defined by the conditions that 
\[
 \dim \, (  E_p  \cap  F_q  )  \,\,  \geq \,\,
 | \, \{ i > p \mid  w(i) \leq q \} \, |    
\]
for all $p$ and $q$.  This dimension condition 
describes the set of points $x$ in $X$ where the fibers
$E_p(x)$ and $F_q(x)$ meet in at least the given dimension, 
a locus locally defined by determinants.  
For the flags in general position (as in the introduction), this locus has 
codimension equal to the length $\ell(w)$  of  $w$, i.e., 
the number of pairs $i<j$ such that $w(i)>w(j)$; equivalently, in notation 
valid for all types,  $\ell(w)$ is the minimum number $\ell$ for which one 
can write $w = s_{i_1} \cdot \ldots \cdot s_{i_{\ell}}$. Note that 
the definition is independent of $n$, and makes sense whenever 
$w(i) < w(i+1)$ for  $i > a$  and  $w^{-1}(j) < w^{-1}(j+1)$ for $j > b$.
For  $i$ and $j$ positive, set
\[
  x_i = c_1(E_{i-1}/E_i)  \,\,\,\, \text{and} \,\,\,\,  y_j = c_1(F_j/F_{j-1}).
\]

Then the class $[\Omega_w]$ of $\Omega_w$ in the cohomology group
$H^{2\ell(w)}(X)$, in the complex setting, or the Chow group 
$A^{\ell(w)}(X)$ in general,  is given by the double Schubert polynomial 
$\frak{S}_w(x,y)$ of Lascoux and Sch\"utzenberger: 
\[
  [\Omega_w]  =  \frak{S}_w(x,y). 
\]
To describe these polynomials, we use the one-line notation, writing 
$w = w(1) \,\, w(2) \, \ldots \, w(m)$  to mean that these are the first 
$m$ values of $w$, and that  $w(i) = i$  for $ i > m$.  Let  $s_i$ be 
the transposition that interchanges $i$ and $i+1$, and let 
$\partial_i$ be the corresponding difference operator with respect to the $x$ variables, 
i.e., 
\[
  \partial_i(f) = (f - s_i(f))/(x_i - x_{i+1});
\]
here  $s_i(f)$ interchanges  $x_i$ and  $x_{i+1}$ in the polynomial $f$, treating $y$ variables as scalars.  
These Schubert polynomials are determined (and can be calculated) 
by the properties that, for  $w = n \, n\!-\!1 \, \ldots \, 2 \,\, 1$, the longest permutation in $S_n$, 
$\frak{S}_w(x,y) = \prod_{i+j \leq n} (x_i - y_j)$.  And, if $w(i) < w(i+1)$,
then  $\frak{S}_w(x,y) = \partial_i(\frak{S}_{w'}(x,y))$, 
where $w' = w\cdot s_i$ is the result of switching the entries in positions 
$i$ and $i+1$.  

There is a natural way to describe degeneracy loci, by specifying 
a decreasing sequence $E_{p_i}$ of bundles in the $E_\bullet$ flag, and an 
increasing sequence of bundles $F_{q_i}$ in the $F_\bullet$ flag, and 
requiring that the corresponding bundles $E_{p_i}$ and $F_{q_i}$ meet 
in at least a given dimension $k_i$.  This idea leads directly to the notion 
of a vexillary permutation. 
Define a {\bf triple of type A}  to be
a triple $\bm{\tau} = (\bf{k}, \bf{p}, \bf{q})$ 
of sequences of positive integers of the same length  $s$, 
\[
  {\bf{k}} \, = \,  k_1 < \dots < k_s, \,\,\,\,\,\, {\bf{p}} \, = \,  p_1 \geq \dots \geq p_s, \,\,\,\,\,\,
{\bf{q}} \, = \, q_1 \leq \dots \leq  q_s,
\]
satisfying the conditions that $k_i \leq q_i$  for all $1 \leq i \leq s$, 
and 
\[
 \tag{*} \,\,\,\,\, (p_i - p_{i+1}) + (q_{i+1} - q_i) \, > \, k_{i+1} - k_i   
\]
for $1 \leq i \leq s-1$, and $p_s - q_s + k_s > 0$.  Setting
$l_i = p_i - q_i + k_i$, these conditions are equivalent to requiring 
$l_i \leq p_i$ for all $i$ and $l_1 > \ldots > l_s > 0$.
The corresponding locus $\Omega_{\bm{\tau}}$ is the locus 
where 
\[
 \dim \, ( \,E_{p_i} \cap F_{q_i} \, ) \,\, \geq \,\, k_i ,
\]
for all $i$ between  $1$ and $s$.   Setting $r_i = q_i - k_i = p_i - l_i$, this is the locus 
where the rank of the map from $F_{q_i}$ to $V/E_{p_i}$ is at most $r_i$, $1 \leq i \leq s$.

This is the same as the locus $\Omega_w$, where  $w = w(\bm{\tau})$ 
is the permutation of minimal length such that the cardinality of 
$\{ j > p_i \mid w(j) \leq q_i \} $ is $k_i$ for all $i$.  To construct this 
permutation, start by putting the  $k_1$ integers $q_1 - k_1 + 1, \ldots q_1$, in 
this order, starting in the first position to the right of $p_1$; then insert $k_2 - k_1$ 
different integers, as large as possible but at most equal to $q_2$, in increasing order, 
as far left as possible but to the right of position $p_2$; continue in this way until $k_s$ 
integers have been entered; then insert the other integers in the other positions, in 
increasing order.  For example, for the triple
$\bm{\tau} = (\, 2 \; 6 \; 8 \, , \, 7 \; 4 \; 2 \, , \, 5 \; 7 \; 9 \,)$, 
we construct its permutation as follows: put $4$ and $5$ in positions $8$ and $9$; 
then put $2$, $3$, $6$, and $7$ in positions $5$, $6$, $7$ and $10$; then put 
$8$ and $9$ in positions $3$ and $4$; and fill in the integers $1$ and $10$ in the first 
two positions, 
getting $w({\bm{\tau}}) = 1 \; 10 \; 8 \; 9 \; 2 \; 3 \; 6 \; 4 \; 5 \; 7$.
The permutations 
arising in this way are vexillary (they avoid the pattern $2 \; 1 \; 4 \; 3$), 
and every vexillary permutation arises from a unique such triple. 

The determinantal formula for this locus (or double Schubert polynomial) is 
particularly easy to describe in this language. 
A triple determines a partition $\lambda = \lambda_1 \geq \dots \geq \lambda_r$ of 
length $r = k_s$.  This is defined by setting $\lambda_{k_i} = p_i - q_i + k_i$, and, for 
general $k$, set 
$\lambda_k = \lambda_{k_i}$ with $k_i$ the minimal integer in the sequence $\bf{k}$ 
that is at least as large as $k$.  

Define 
\[
 a(k_i) \, =\, \prod_{j = 1}^{p_i} (1 + x_j) / \prod_{j = 1}^{q_i} (1 + y_j) ,
\]
and, for $k \leq r$ not equal to any $k_i$,  set $a(k) = a(k_i)$,  where  $k_i$  is the 
smallest in the sequence that is larger than $k$.  Then the formula for this 
locus, which is  $\frak{S}_{w(\bm{\tau})}(x,y)$, is the multi-Schur determinant 
\begin{equation}\label{e.schur}
 s_\lambda(a(1), \dots , a(r)),
\end{equation}
which is defined to be the determinant of the $r$ by $r$ matrix whose $(i,j)$ entry is 
$ a(i)_{\lambda_i + j - i} $.

There is a natural duality in this story.  In the general setup, let $V' = V^*$ be the dual vector bundle, 
and let  $E'_i = \Ker(V^* \to F_i^*)$,  $F'_j = \Ker(V^* \to E_j^*)$.  Then the locus $\Omega_w$ for 
the original setup is equal to the locus $\Omega_{w^{-1}}$ for this setup.  This reflects (and reproves) 
the fact that  $\frak{S}_{w^{-1}}(-y,-x) =  \frak{S}_w(x,y)$.
Define the dual of $\bm{\tau}$ to be  the triple 
\[
 {\bm{\tau}}^* \, = \, (l_s < \dots < l_1, \,\, q_s \geq \dots \geq q_1, \,\,  p_s \leq \dots \leq p_1). 
\]
Then  $w(\bm{\tau}^*) = w(\bm{\tau})^{-1}$.  The partition for ${\bm{\tau}}^*$ is the conjugate 
(transpose) $\lambda'$ of the partition for $\bm{\tau}$.  
The corresponding Schur polynomial, $s_{\lambda'}(a'(1), \dots, a'(r'))$
with $r' = l_1$, using the rows with 
\[
 a'(l_i) \, =\, \prod_{j = 1}^{q_{s+1-i}} (1 - y_j) / \prod_{j = 1}^{p_{s+1-i}} (1 - x_j) ,
\]
and $a'(l) = a'(l_i)$ for $l_{i-1} < l < l_i$.  This is equal to the original Schur determinant, 
a basic identity among Schur determinants (which also follows from this geometry).

There is a natural notion of {\bf redundant triple} in type A, which allows the inequalities between 
the $l$'s to be weak: $l_1 \geq \dots \geq l_s$.  These give 
the same vexillary permutations, and the same determinants, although the entries of the 
matrices change.  (The reason is that one matrix can be changed to the other by a sequence of 
elementary row operations.)  One can reduce a redundant triple to an ordinary triple by leaving out 
any $(k_i,p_i,q_i)$ for which $l_i = l_{i+1}$.  For example, the triple 
\[
 (1\;2\;3\;4\;5\;6\;7\;8,\,\, 7\;7\;6\;6\;5\;4\;3\;2\,, \,\, 4\;5\;6\;7\;7\;7\;9\;9)
\] 
reduces to the above triple 
$( 2 \; 6 \; 8, \,\, 7 \; 4 \; 2, \,\, 5 \; 7 \; 9 )$.
The added conditions on dimensions of intersections follow from those of 
the reduced triple.  
Such redundant conditions appear often in the literature, including the Kempf-Laksov formula.\footnote{This completes and corrects Remark 9.16 of \cite{F2}; note that the notion of redundancy is not preserved by the duality that interchanges $\bf{k}$ and $\bf{l}$.}  
In addition, it can be useful to 
start with an ordinary triple, and expand it to a redundant triple, as in this example, with $k_i = i$; in this case $\lambda = \bm{l}$.  Note that this can often be done in more than one way.

Here is a proof of the formula \eqref{e.schur} for the locus $\Omega_\triple$.
We may replace $X$ by the bundle of flags in $V$, as described in the introduction, and assume the bundles are in general position.  
Consider first the case where $q_i = k_i = i$ for $1 \leq i \leq s$.  The locus is defined first by 
requiring  $F_1$ to be contained in $E_{p_1}$, which is defined by the vanishing of 
a section of $\Hom(F_1,V/E_{p_1})$, so is given by the top Chern class 
$c_{p_1}( F_1^* \otimes V/E_{p_1} )$. 
On this locus one has the locus where the map from $F_2/F_1$ to $V/E_{p_2}$ vanishes, given 
by  $c_{p_2}((F_2/F_1)^* \otimes V/E_{p_2})$.  Continuing in this way, one sees that the locus 
is given by the formula
\[
 \prod_{i=1}^s c_{p_i}((F_i/F_{i-1})^* \otimes V/E_{p_i}).
\]
It is an elementary calculation (cf.~\cite{F1}, App.~A.9 or \cite{Mac} (3.6)),
using row operations on the matrices, to see that this is equal to the Schur polynomial
\[
 s_{p_1 \ldots p_s} (c(V - E_{p_1} - F_{1}), \ldots , c(V - E_{p_s} - F_{s})),
\]
and this is the required formula in this case. 

Next consider the case where $k_i = i$ for all $i$, but with $q_i$ arbitrary.  We construct the flag bundle $Y$ over $X$ of flags $D_1 \subset D_2 \subset \dots \subset D_s$, with $D_i$ of rank $i$ and contained in $F_{q_i}$.  
The locus $\Omega_{\bm{\tau}}$ on $X$ is the birational image of the locus on 
this flag variety given by the conditions that $D_i$ is contained in $E_{p_i}$.  By the case 
just considered, this is given by the polynomial $s_{p_1 \ldots p_s} (c(1), \ldots , c(s))$, with 
$c(i) = c(V - E_{p_i} - D_i)$.  To conclude in this case, we need a Gysin formula to push this 
forward by the projection $\pi \colon Y \to X$.  This formula, due  to \cite{JLP}, 
and given in this form in \cite{F1}, Cor.~14.2, is
\begin{equation*}
\begin{split}
 \pi_* ( s_{p_1 \ldots p_s} (c(V - E_{p_1} - D_1), \ldots , c(V - E_{p_s} - D_s)) )  = \\
s_{\lambda_1 \ldots \lambda_s} (c(V - E_{p_1} - F_{q_1}), \ldots , c(V - E_{p_s} - F_{q_s})),
\end{split}
\end{equation*}
where $\lambda_i = p_i - q_i + i$.  This is the required formula in this case. 
The general case is reduced to the preceding case by inserting integers between successive 
$k_i$ and $k_{i+1}$, and corresponding integers between $p_i$ and $p_{i+1}$ and between 
$q_i$ and $q_{i+1}$, to get a redundant triple with $k_i = i$ for all $i$.   

Standard arguments, by passing to a flag bundle over $X$ (see \cite{F2}), show that the formula is valid if 
one only has the bundles $F_{q_1} \subset F_{q_2} \subset  \dots$ and $E_{p_1} \supset 
E_{p_2} \supset \dots$ of  $V$.

Note that the proof of this formula includes a proof that the double Schubert polynomials of 
vexillary permutations are given by the corresponding determinants -- a result proved 
in \cite{Mac} with a fair amount  of algebra.  It also implies that construction of general double Schubert polynomials is independent of the choice of $n$ used in their definition.

\section{Type C, Symplectic Bundles}

We begin with two algebraic lemmas, which can be found in a geometric context in \cite{K}.
The fact that the matrices whose Pfaffians are the $\Q_\lambda$ is skew-symmetric is 
a special case of the following lemma.

\begin{lemma}\label{l.skewQkl}
Let $\lambda_1,  \dots   \lambda_r$ be positive integers.  For 
$1 \leq k \leq r$ let $S(k)$ be a set of positive integers of cardinality strictly 
less than $\lambda_k$.  Set $q(k) = Q \cdot \prod_{j \in S(k)} (1 + t_j)$.  Then the 
matrix whose $(k,l)$ entry is 
\[ q(k)_{\lambda_k} \cdot q(l)_{\lambda_l}  + 2 \sum_{j=1}^{\lambda_l} (-1)^j 
q(k)_{\lambda_k + j} \cdot q(l)_{\lambda_l - j}
\]
is skew symmetric. 
\end{lemma}

\begin{proof} Since $Q \cdot Q^* = 1$, it follows that $q(k) \cdot q(l)^* = 
\prod_{j \in S(k)} (1 + t_j) \cdot \prod_{j \in S(l)} (1 - t_j)$.  This is a 
polynomial of degree strictly less than $\lambda_k + \lambda_l$.  
The assertion that its term of degree $\lambda_k + \lambda_l$ vanishes 
is equivalent to the assertion of the lemma.
\end{proof}

We next show that the polynomial ${\bf{Q}}_{\bm{\tau}}(x,y)$ defined in the introduction from a redundant triple is the same as that of the actual triple obtained from it by omitting the redundant terms.

\begin{lemma}  Let $\bm{\tau} = (\bf{k}, \bf{p}, \bf{q})$  be a redundant triple, and 
let $\bm{\tau'}$ be the (possibly redundant) triple obtained from it by omitting a term $k_i, p_i, q_i$ for 
which $p_i - p_{i+1} + q_i - q_{i+1} = k_{i+1} - k_i$.  Then ${\bf{Q}}_{\bm{\tau}}(x,y) = {\bf{Q}}_{\bm{\tau'}}(x,y)$.
\end{lemma}
\begin{proof} 
For this it is useful, as in \cite{K}, to note that Pfaffians 
\[
\Pf_{\lambda_1 \ldots \lambda_r}(c(1), \ldots, c(r))
\] 
can be defined by the same formulas for an arbitrary sequence $\lambda_1, \ldots , \lambda_r$ of integers; this Pfaffian is skew-symmetric in these indices, and vanishes when two are equal.   The lemma follows from the multi-linearity of the Pfaffians, together with the fact that 
they vanish whenever two indices are equal.  In fact, one verifies that if $c(k)  = c(k+1) = \dots = c(k+m)$, and $\lambda_i = \lambda_{i+1} + 1$ for $k \leq i < k+m$, then the Pfaffian $\Pf_{\lambda_1 \ldots \lambda_r}(c(1), \ldots, c(r))$ does not change if $c(k)$ is multiplied by a polynomial $\prod_{j=1}^m(1+z_j)$ (cf.~\cite{K}, proof of Thm. 1.1).
\end{proof}

We describe in more detail how a triple $\bm{\tau} = (\bf{k}, \bf{p}, \bf{q})$ 
determines a signed permutation $w(\bm{\tau})$, written in one-line notation.  
We want $w$ minimal with $\#\{a \geq p_i \mid w(a) = \bar{b}, b \geq q_i\} = k_i$.
Start by inserting a sequence of $k_1$ consecutive 
barred integers, $\bar{q_1+k_1 - 1}$ to $\bar{q_1}$, in positions  $p_1$ to $p_1-k_1+1$. 
Then insert $k_2 - k_1$ barred elements, bars of the smallest remaining integers that are 
at least $q_2$, placed in order (bars of decreasing sequences, or 
increasing if one regards barred 
integers as negative) at or to the right of position $p_2$ but as close 
as possible to that position.  
Continue in this way until one has entered $k_s$ barred integers.  
Conclude by inserting the remaining positive integers (that haven't been entered in barred form), in increasing order, in the remaining spots.  
For example, consider the triple $\bm{\tau}$ with ${\bf{k}} = 2 \; 3 \; 5 \; 8$, ${\bf{p}} = 8 \; 6 \; 6 \; 2$, and ${\bf{q}} = 6 \; 5 \; 2 \; 2$. One puts  $\bar{7}$ and $\bar{6}$ in positions $8$ and $9$, then $\bar{5}$ in position $6$, then $\bar{3}$ and $\bar{2}$ in positions $7$ and $10$, then $\bar{9}$, $\bar{8}$, and $\bar{4}$ in positions $2$, $3$, and $4$; finally put $1$ and $10$ in the remaining positions $1$ and $5$, arriving at $w(\bm{\tau}) = 1 \; \bar{9} \; \bar{8} \; \bar{4} \; 10 \; \bar{5} \; \bar{3} \; \bar{7} \; \bar{6} \; \bar{2}$. 

One can reconstruct the triple from the vexillary signed permutation.  To start, $p_1$ is the position after the last descent; starting at that position one finds a sequence of $k_1$ consecutive barred integers, ending in $\bar{q_1}$. 
One repeats this on the remaining, but regarding elements as being in sequence if the only integers missing have already been taken account of.  The combinatorics of this correspondence will be worked out in detail in \cite{AF2}.  

We turn next to the direct proof of Theorem 1.  The proof follows the lines 
in type A, but using a Gysin formula of Kazarian's.  
We use the same flag bundle, with tautological flags $D_1 \subset 
\dots \subset D_s$, with $D_i$ of rank $i$ and $D_i \subset F_{q_i}$. 
We have a sequence of projective bundles $X_s \to   \dots \to X_1 \to X$, 
where $X_k$ is the corresponding flag bundle with $s$ replaced by $k$.  Using 
the standard convention of not changing notation for pullback bundles, the 
projection $\pi_k \colon X_k \to X_{k-1}$ is the projective bundle 
$\P(F_{q_k}/D_{k-1})$, with its tautological line bundle $D_k/D_{k-1}$.
On $X_k$ we have the locus $Z_k$ where $D_k$ is contained in $E_{p_k}$.  
The locus $Z_s$ maps birationally onto the locus $\Omega_{\bm{\tau}}$, so our goal is to compute the push-forward of its class $[Z_s]$.  This is done in stages.  We let $\phi_k \colon Z_k \to Z_{k-1}$ be the map induced by $\pi_k$, and set $\phi \colon Z_s \to X$ to be their composite.
 
On $X_1$, $Z_1$ is given by the vanishing of a section of $\Hom(D_1,V/E_{p_1})$; so $[Z_1] = 
c_{n+p_1-1}(D_1^* \otimes V/E_{p_1})$.  
On the restriction $X_2 |_{Z_1}$, we have $D_1 \subset E_{p_1} \subset E_{p_2} \subset D_1^\perp$.
The locus $Z_2$ is given by the vanishing of a section of $\Hom(D_2/D_1,D_1^\perp/E_{p_2})$; so $[Z_2] = c_{n+p_2-2}((D_2/D_1)^* \otimes D_1^\perp/E_{p_2})$.  Continuing, on $X_k |_{Z_{k-1}}$ we have $D_{k-1} \subset E_{p_{k-1}} \subset E_{p_k} \subset D_{k-1}^\perp$, and the locus $Z_k$ where $D_k$ is contained in $E_{p_k}$ is given by the formula $[Z_k] = 
c_{n+p_k-k}((D_k/D_{k-1})^* \otimes D_{k-1}^\perp/E_{p_k})$.  (See [F3] and [FP] for a 
discussion of why these equations cut out these loci with their reduced strictures.)

We need a Gysin formula to push forward these Chern classes.  Kazarian has proved 
such a formula in \cite{K}, in the special case where all $E_{p_i}$ are equal, but exactly the 
same argument works in our more general setting.  Let $h_i = - c_i(D_i/D_{i-1})$, and set
\[
c(k) = c(V - E_{p_k} - F_{q_k}) = c(V - E - F) \cdot c(E/E_{p_k})\cdot c(F/F_{q_k})
\]
the assertion is that, for any nonnegative integers 
$m_1, \dots, m_s$, 
\begin{equation}
\phi_*(h_1^{m_1} \cdot \dots\cdot  h_s^{m_s}) = \Pf_{\lambda_1+m_1, \dots \lambda_s + m_s}(c(1), \dots c(s)). 
\end{equation}
Taking all $m_i = 0$ gives the asserted formula of the theorem.  This assertion is based on 
the following lemma.  Fix $k$ between $1$ and $s$, and set 
\[
H^{(k)} = c(D_{k-1} - D_{k-1}^*) = \prod_{j=1}^{k-1} \frac{1-h_i}{1+h_i} = \prod_{j=1}^{k-1} (1 - 2h_j + 2h_j^2 - 2h_j^3 + \dots ).
\]

\begin{lemma}\label{l.typeCstep}
For any nonnegative integer $m$, 
\[
(\phi_k)_*(h_k^m) = \sum_{j=0}^{\lambda_k+m} H^{(k)}_j c(k)_{\lambda_k+m - j}.
\]
\end{lemma}
A key to this is to write, on $X_k$, 
\begin{equation*}
\begin{split}
 h_k^m \cdot [Z_k] &=h_k^m \cdot c_{n+p_k-k}((D_k/D_{k-1})^* \otimes D_{k-1}^\perp/E_{p_k})  \\
 &= c_{n+p_k-k+m}(D_{k-1}^\perp/E_{p_k} - D_k/D_{k-1})  \\
 &= c_{n+p_k-k+m}(D_{k-1}- D_{k-1}^* + V - E_{p_k} - F_{q_k} + F_{q_k}/D_k).
\end{split}
\end{equation*}

All terms in the last equation except the last are pullbacks from $Z_{k-1}$, and the last 
is the tautological quotient bundle from the projective bundle $X_k \to X_{k-1}$.  The standard 
Gysin formula for projective bundles, this one of relative dimension $n+q_k - k$, implies that
\begin{equation*}
\begin{split}
(\phi_k)_*(h_k^m) &= c_{n+p_k-k+m - (n+1-q_k - k)}(D_{k-1}- D_{k-1}^* + V - E_{p_k} - F_{q_k} ) \\
&=  c_{\lambda_k+m}(D_{k-1}- D_{k-1}^* + V - E_{p_k} - F_{q_k}).
\end{split}
\end{equation*}
With this, the proof of (1) is reduced to some formal algebra, which is given in Appendix C of \cite{K}.  It also follows from Proposition~\ref{p.gysin} below, by removing the powers of $1/2$, setting the $u_k$'s to $0$, the $g(k)$'s to $1$, and $d(k)=c(k)$.  (In this special case, the formula and its proof are exactly the same as in \cite{K}.)

\smallskip

As indicated in the introduction, it is also possible to deduce the theorem from Kazarian's result for 
Lagrangian subbundles.  For this, set $V' = V \oplus V$, and put a bilinear form on $V'$ by the formula 
\[
\langle \, (u,v), (u',v') \, \rangle = \langle u, u' \rangle - \langle v , v' \rangle, 
\]
where the right side uses the 
given bilinear form on $V$.  This is skew-symmetric and non degenerate, and the diagonal $\Delta_V$ is 
a Lagrangian subbundle, as is $E \oplus F$.  The locus where $\Delta_V$ meets each $E_{p_i} \oplus F_{q_i}$ in dimension at least $k_i$ is the 
same as the locus where each  $E_{p_i}$ meets $F_{q_i}$ in dimension at least $k_i$.  Note also that 
\[
c(V \oplus V) / c(\Delta_V) c(E_{p_i} \oplus F_{q_i}) = c(V)/c(E_{p_i})c(F_{q_i}),
\]
from which the Theorem follows.  Note that this geometry reflects the fact that the polynomial $\Q_{\bm{\tau}}(x,y)$ is 
obtained from the polynomial $\Q_{\lambda(\bm{\tau})}(t)$ by appropriate substitutions.

\smallskip

Although not needed for our theorems, we include here the proof that our polynomials $\Q_\lambda$ agree with those of \cite{IMN}, up to sign. 

\begin{proposition}
For any strict partition $\lambda$, the polynomial $\Q_\lambda$ is equal to the 
polynomial denoted $Q_\lambda(x|-t)$ in {\rm\cite{IMN}}. 
\end{proposition}

\begin{proof}
This is true by definition (see \cite{IMN} Remark 4.4) when $\lambda$ has length $1$.  When 
$\lambda$ has length greater than $2$, both are equal to the same Pfaffian of a 
matrix whose entries are those of length $2$ (see \cite{IMN}, \S9).  So it suffices to prove that $\Q_{k \, l} = Q_{k \, l}(x|-t)$ 
for any $k > l > 0$.  This can be done by the method of Ivanov [Iv], as follows.  Since the  $Q_\lambda(x|-t)$ form
a basis for $\Gamma$ over $\Z[t]$, we may write $\Q_{k \, l} = \sum a_\lambda Q_\lambda(x|-t)$, the 
sum over strict partitions $\lambda$ of length at most $k+l$, with $a_\lambda$ homogeneous of degree $k+l-|\lambda|$ in $\Z[t]$.  Since the top terms of both are $Q_{k \, l}$, we know that $a_{k \, l} = 1$. 
We use the fact (\cite{IMN} Proposition 4.2) that for any strict partition $\nu$ of length $r$, $Q_\lambda(-t_{\nu_1}, -t_{\nu_2}, \ldots, -t_{\nu_r}, 0, \ldots|-t)$ vanishes if $\lambda \nsubseteq \nu$, and does not vanish if $\lambda = \nu$.  If the assertion is false, there is a coefficient $a_\lambda$ that is not zero, with $|\lambda|$ minimal and less than $k+l$.  Evaluating with $\nu = \lambda$  would give a 
nonzero result.  The following lemma shows that this is impossible. 
\end{proof}

\begin{lemma}
For any strict partition $\nu = \nu_1 > \dots > \nu_r > 0$ with $\nu_1 < k$ or $\nu_2 < l$, the specialization $Q \mapsto \prod_{i=1}^r (1-t_{\nu_i})/(1+t_{\nu_i})$ sends $\Q_{k \, l}$ to $0$.
\end{lemma}
\begin{proof}
This is an elementary calculation.  For example, if $\nu_1 < k$, then $c(k)$ maps to the product 
of all $(1 \pm t_j)$, for $j < k$, with the sign negative if $j$ is equal to some $\nu_i$ and positive otherwise. 
Since this is a polynomial of degree less than $k$, all $c(k)_{k+j} = 0$ for $j \geq 0$, so $Q_{k\, l}$ maps to $0$. 
The calculation when $\nu_1 \geq k > l > \nu_2$ is a slightly longer but elementary exercise.
\end{proof}

To complete the discussion of double Schubert polynomials in type C, one needs to make the 
Weyl group and difference operators act on the ring $\Gamma[x,y]$, compatibly with the actions 
on flag bundles.  We describe the actions for the $x$ variables, with those for the $y$ variables 
obtained by interchanging $x$'s and $y$'s.  The operators $s_i$ and $\partial_i$ are the same 
as in type A for $i \geq 1$, 
with $s_i$ fixing the $Q_k$'s as well as the $y$'s.  The operator $s_0$ takes $x_1$ to $-x_1$, 
and fixes the other $x_i$'s, but it takes $Q_k$ to $Q_k + 2\sum_{j=1}^k x_1^j Q_{k-j}$.  Equivalently, 
$s_0$ takes $Q$ to $(1+x_1)/(1-x_1)\cdot Q $, which is compatible with the specialization to 
$c(E^*)/c(F) = \prod_{i=1}^n(1-x_i)/c(F) \mapsto (1+x_1)\prod_{i=2}^n(1-x_j)/c(F) =  (1+x_1)/(1-x_1)\cdot c(E^*)/c(F) $.
The difference operator $\partial_0$ is defined by  $\partial_0(f) = (f - s_0(f))/(-2x_1)$; in particular, $\partial_0(Q_k) = \sum_{j=1}^k x_1^{j-1}Q_{k-j}$.   (See \cite{BH} and \cite{IMN} for more about these actions.)  

\smallskip
As an application, generalizing results of \cite{K} and \cite{LP},  
let $E$ be a vector bundle, with a flag of 
subbundles $E = E_1 \supset E_2 \supset \dots $, with $E_p$ of 
corank $p-1$ in $E$.  Let $F$ be a subbundle of $E$, of corank $m$. 
Let $\phi \colon F \to E^*$ be a symmetric map of vector bundles, 
which means that the induced composition $F \to E^* \to F^*$ is equal to 
its dual.  Given sequences of positive integers $k_1 < \dots < k_s$ 
and $p_1 > \dots > p_s$, with $p_i - p_{i+1} > k_{i+1} - k_i$, we 
have the locus where the kernel $\Ker(\phi)$ meets $E_{p_i}$ in 
dimension at least $k_i$.  Define $\lambda$ to be the minimal strict 
partition with  
$\lambda_{k_i} = p_i + m$.  Let $c(k_i) = c(E^* - F)\cdot c(E/E_{p_i})$, 
 set $c(k) = c(k_i)$ for $k_i \geq k$ minimal in general, and let $r = k_s$.

\begin{cor*}
The class of the locus where $\Ker(\phi)$ meets $E_{p_i}$ in 
dimension at least $k_i$ is the Pfaffian $\Pf_\lambda(c(1), \dots ,c(r))$.
\end{cor*}

The proof is by setting $V = E \oplus E^*$, with its standard skew-symmetric 
form: $\langle \, (e,\alpha), (e',\alpha') \, \rangle = \alpha(e') - \alpha'(e) $.  Then $E = E \oplus 0$ 
is a Lagrangian subbundle of $V$.  Let $F_q$ be the graph of $\phi$, 
with $q = m+1$.  Then $F_q$ is isotropic with $\rank(F_q) = \rank{E} +1 - q$.  The locus 
where $F_q$ meets $E_{p_i}$ in dimension at least $k_i$ is the required locus.
The formula is a special case of the main theorem, taking $q_i = q$ for all $i$. 

As a special case, the locus where the kernel of $\phi$ meets a given $E_p$ in dimension 
at least $k$ is given by $\Pf_\lambda(c, \dots, c)$, where $\lambda = (p+m+k-1, \dots ,p+m)$
and $c = c(E^* - F)$.  In particular, the locus where a symmetric 
map $E \to E^*$ drops rank by at least $k$ is given by the polynomial $\Pf_{(k, \dots, 1)}(c, \dots, c)$, 
with $c = c(E^* - E)$.

\section{Type B, Orthogonal Bundles of Odd Ranks}

The polynomials in types B and D are based on the ring 
\[
\Gamma' = \Z[P_1, P_2, \dots]/(P_k^2 + 2\sum_{j=1}^{k-1} (-1)^j P_{k+j}P_{k-j} + (-1)^kP_{2k},  \,\, k = 1, 2, \dots) .
\]
One can identify $P_k$ with $\frac{1}{2}Q_k$, so $\Gamma$ becomes a subring of $\Gamma'$; but note that 
$P_0$ and $Q_0$ are both $1$, so $P = 1 + P_1 + \dots$.  With this identification, for each strict partition $\lambda$, one can set $P_\lambda = 2^{-r}Q_\lambda$, where $r$ is the length of $\lambda$.  Equivalently, 
 $P_{k\, l} = P_kP_l + 2\sum_{j=1}^{l-1} P_{k+j}P_{l-j} + (-1)^lP_{k+l}$, and $P_\lambda$ is the Pfaffian of the matrix with $(i,j)$-entry $P_{\lambda_i \, \lambda_j}$.

There are three different bases for the ring $\Gamma'[t]$, each parametrized by strict partitions, one for type B and two for type D.  
We describe here the one for type B, which we denote by $\P_\lambda$.\footnote{In the notation of \cite{IMN}, these would be written $P_\lambda(x|0,-t)$.}  These can be defined by the formula  $\P_\lambda = 2^{-r}\Q_\lambda$, with $r$ the length of $\lambda$.  
Equivalently, setting $b(k) = p(\lambda_k)$, where $p(k) = P\cdot \prod_{j=1}^{k-1}(1+t_j)$, we have $\P_\lambda = \Pf_{\lambda}(b(1), \dots, b(r))$.  

Triples are defined exactly as in type C.  For a triple $\bm{\tau}$, set $b(k_i) = P \cdot \prod_{j=1}^{p_i - 1} (1 + x_j) \cdot \prod_{j=1}^{q_i - 1} (1 + y_j)$, 
and set $b(k) = b(k_i)$, where $k_i$ is minimal greater than or equal to $k$. 
Then ${\bf{P}}_{\bm{\tau}}(x,y)$ is defined to be the Pfaffian $\Pf_{\lambda}(b(1), \dots, b(k_s))$.  As in type C, these can be defined by substituting the $x$'s and $y$'s for the $t$'s in $\P_\lambda$, in the same orders.  The definition of the signed permutation $w(\bm{\tau})$ of a triple is defined by the same prescription as in type C.

Exactly as in type C, these polynomials can be used to define all the double Schubert polynomials $\frak{B}_w$ of type B, 
starting from the polynomial for $w_\circ^{(n)}$ being ${\bf{P}}_{\bm{\tau}}(x,y)$, where $\bm{\tau}$ is the triple with ${\bf{k}} = 1 \, 2 \, \dots \, n$ and  ${\bf{p}} = {\bf{q}} = n \, n\!-\!1 \, \dots \, 1$.
The only difference is that the difference operator $\partial_0$ is defined by the formula  $\partial_0(f) = (f - s_0(f))/(-x_1)$.

The geometric setting for type B is a vector bundle $V$ of odd rank $2n+1$, together with a nondegenerate symmetric bilinear form on $V$ (here with values in the trivial line bundle).  We assume $V$ comes with two flags of isotropic subbundles, $E = E_1 \supset E_2 \supset \dots$, and $F = F_1 \supset F_2 \supset \dots$, with $E_i$ and $F_i$ having ranks $n+1-i$, as in 
type C.  Degeneracy loci $\Omega_{\bm{\tau}}$ are also defined as in type C, requiring $\dim(E_{p_i} \cap F_{q_i})\geq k_i$, 
for the corresponding triple.
For general signed permutations $w$, one needs to add the bundles $F_{-i} := F_{i+1}^\perp$ for $i \geq 0$; then $\Omega_w$ is the locus where 
the dimension of $E_p \cap F_q$ is  at least the number of $a \leq p$ such that $-w(a) \geq q$, for $p \in \{1, \dots ,n\}$ and $q \in \{\pm 1, \dots ,\pm n\}$.  
There is one important difference, however.  In this situation, even when the bundles are in general position, the determinants that locally specify rank 
conditions do not cut out these degeneracy loci with their reduced structures; in fact, the factors of $2$ that occur come from the multiplicities with which they occur.  We refer to [F3] and [FP] for a discussion of these points.  We will continue to 
write loci by these inequalities, but they need to be understood to be the reduced subschemes defined by the corresponding equations.

We define $E_0$ to be $E^\perp$, and $E_{-i} = E_{i+1}^\perp$ for $i \geq 1$.  The bilinear form determines an isomorphism of $V/E_0$ with the dual $E^*$ of $E$, and $V/F_0 \cong F^*$.
From the filtration $E_1 \subset E_0 \subset V$ we have an isomorphism $\bigwedge^{2n+1}V = \bigwedge^nE_1 \otimes E_0/E_1 \otimes \bigwedge^n(V/E_0)$, implying that $\bigwedge^{2n+1}V \cong E_0/E_1$; since the same is true for   $F_\bullet$, we have an isomorphism $E_0/E_1 \cong F_0/F_1$.  This implies that 
\[
c(V - E_0 - F_1) = c(V - E_1 - F_0),
\]
and it is to this that $Q$ will specialize.  To specialize $P$, we need the basic fact, based on Edidin and Graham \cite{EG}, that these Chern classes $c_i(V-E_0 - F_1)$ are canonically divisible by $2$, for $i > 0$.  In \cite{EG} classes $a_i$ are produced so 
that $2a_i = c_i(E) - c_i(F)$, for $1 \leq i \leq n$.  
Then $c(V - E_0 - F_1) = c(E^*)/c(F) = (c(F^*) + 2\sum_{i=1}^n(-1)^i a_1)/c(F)$.  
It therefore suffices to note that, for any bundle $F$, $c(F^*)/c(F)$ has all terms of positive degree canonically divisible by $2$, which follows from the identity $(1-y_i)/(1+y_i) = 1 - 2y_i + 2y_i^2 + \dots$.  (To see that there is a canonical division by $2$ of these classes, one can also, as in \cite{F3},  reduce to the situation where there is no $2$-torsion.)

We can therefore specialize the $P_k$ to these classes, which we denote for simplicity by $\frac{1}{2}c_k(V - E_0 - F_1)$. 
Set $x_i = c_1(E_i/E_{i+1})$ and $y_i = c_1(F_i/F_{i+1})$.  

\begin{thm}\label{t.theorem2}
{\rm (1)} The class of $\Omega_{\bm{\tau}}$ is obtained from ${\bf{P}}_{\bm{\tau}}(x,y)$ by this substitution. 

\smallskip
\noindent
{\rm (2)} The polynomial ${\bf{P}}_{\bm{\tau}}(x,y)$ is equal to the double Schubert polynomial 
$\frak{B}_w(x,y)$, where $w = w(\bm{\tau})$.
\end{thm}

Note that, for any signed permutation $w$,  $\frak{B}_w = 2^{-r(w)} \frak{C}_w$, where $r(w)$ is the number of 
barred values in its one-line notation.  As in type C, the polynomial $\frak{B}_w$ is the class of the locus 
$\Omega_w$, for flags of bundles in general position.

The proof is almost identical to that in type C.  As before, the problem is reduced to the case where $\bk = 1 \, \dots \, r$. The same construction leads to bundles $X_k$ with loci $Z_k$ defined by the vanishing of sections of bundles.  The only difference is that these sections vanish doubly along the loci $Z_k$.  This leads 
to the factors of $2^{-r}$, since there are $r$ such bundles.

The proof using the diagonal trick must be modified from type C, as $V \oplus V$ would take us from type B 
to type D.  To do this, one uses a splitting principle to get to the situation where $V = U \oplus M$ where $U$ is 
a bundle of rank $2n\!-\!1$ with a symmetric form, and $M$ is a line bundle whose symmetric form is given by an isomorphism of $M^{\otimes 2}$ with the trivial bundle.  Then $E$ and $F$ are subbundles of $U$.  One then takes 
$V' = U \oplus U \oplus M$, with pairing $\langle \, (u,v,z), (u',v', z') \, \rangle = \langle u, u' \rangle -  \langle v, v' \rangle + \langle z, z' \rangle$.  Then $\Delta_U \oplus 0$ and $E \oplus F \oplus 0$ are maximal isotropic subbundles, and 
one proceeds as before.  Note also that the Chern class formulas agree, since 
\[
c(V')/c(\Delta_U\oplus M)c(E\oplus F) = c(V)/c(E\oplus M) c(F) = c(V)/c(E_0) c(F).  
\]

\section{Type D, Orthogonal Bundles of Even Ranks}

Type D differs from type B geometrically because of the existence of two families 
of isotropic subspaces of a vector space with a nondegenerate symmetric bilinear 
form.  In order to treat both families, there will be ``even" and ``odd" polynomials. 
To describe them uniformly, we introduce the notion of a {\em partition of type D}, 
by which we mean sequence 
\[
\lambda = \lambda_1 > \lambda_2 > \dots > \lambda_r \geq 0
\]
The {\em length} of $\lambda$ is defined to be $r$.  Each partition of type D 
determines a strict partition, which is itself if $\lambda_r$ is positive, and 
$(\lambda_1, \lambda_2 , \dots , \lambda_{r-1})$ if $\lambda_r = 0$.  
Each strict partition comes from two partitions of type D, one of even length, 
and one of odd length.  For example $(4,2,1,0)$ has even length, $(4,2,1)$ has 
odd, both with underlying strict partition $(4,2,1)$.  We write $Q_\lambda$ and 
$P_\lambda$ for the elements in $\Gamma$ and $\Gamma'$ determined by 
the underlying strict partition of a partition $\lambda$ of type D.

We will define a type D polynomial $\R_\lambda = \R_\lambda(t)$ in $\Gamma'[t]$ 
for each partition $\lambda$ of type D.  Those for $\lambda$ of even length, as 
well as those for $\lambda$ of odd length, will form a basis for $\Gamma'[t]$  over 
$\Z[t]$.  Those of even length will be equal to the polynomials appearing in \cite{IMN}, 
although our description is different.   
Let $e(k) =  \prod_{i=1}^k(1+t_i) = 1 + e(k)_1 + \dots + e(k)_k$, so $e(k)_i$ is the 
$i^{\text th}$ elementary symmetric polynomial in $t_1, \dots, t_k$.  Let  
$r(k) = Q \cdot e(k)$.  Define 
\[
\R_k = \frac{1}{2}(r(k)_k + e(k)_k) = P_k + e(k)_1P_{k-1} + \dots + e(k)_k,
\]
and define
\begin{equation*}
\begin{split} 
\R_{k \, l} &= \frac{1}{4} \bigl( (r(k)_k - e(k)_k)(r(l)_l+e(l)_l) + 2 \sum_{j=1}^l (-1)^j r(k)_{k+j}r(l)_{l-j} \bigl) \\
&= (P_k + e(k)_1P_{k-1} + \dots + e(k)_{k-1}P_1)(P_l + e(l)_1P_{l-1} + \dots + e(l)_{l})  \\
&\,\,\, +   \sum_{j=1}^l (-1)^j (P_{k+j} + \dots + e(k)_kP_j)(2P_{l-j} +2e(l)_1P_{l-j-1} + \dots + e(l)_{l-j}).
\end{split}
\end{equation*}
These polynomials are in $\Gamma'[t]$, and the matrix $\R_{k \, l}$ is skew-symmetric: 
$\R_{k \, k} = 0$ and $\R_{k \, l}+ \R_{l \, k} = 0$.  This follows from the equation
\[
\sum_{j\in \Z} (-1)^j r(k)_{k+j} r(l)_{l-j} = e(k)_k e(l)_l \,\, ,
\]
which follows from the fact that $r(k)\cdot r(l)^* = e(k)\cdot e(l)^*$.

Note that $\R_{k\, 0} = \frac{1}{2}(r(k)_k - e(k)_k) = P_k + e(k)_1P_{k-1} + \dots + e(k)_{k-1}P_1$, 
which is not equal to $\R_k$.  Note also that $\R_0 = 1$ (as is the polynomial for the empty partition).

From this skew-commutativity, we can define $\R_\lambda$ to be the Pfaffian of the matrix whose $(i,j)$ entry is $\R_{\lambda_i \,\lambda_j}$.  When $r$ is odd, this is interpreted as before:
\begin{equation}\label{e.oddpf}
  \R_\lambda = \sum_{k=1}^r (-1)^{k-1} \R_{\lambda_k}\, \R_{\lambda_1,\ldots,\hat{\lambda_k},\ldots,\lambda_r}.
\end{equation}
Note that one cannot, as in types B and C, reduce the odd case to the even case by adjoining 
a $0$ to the partition, as we have seen that $\R_k \neq \R_{k\,0}$.

The leading term of $\R_\lambda$ is $P_\lambda$, which implies that both the polynomials for $\lambda$ of even length and those of odd length form a basis of $\Gamma'[t]$ over $\Z[t]$.

These polynomials, for $r$ even, agree with those in \cite{IMN}; this means that our $\R_\lambda$ is equal to their $P_\lambda(x | -t)$ (where the subscript $\lambda$ to the $P$ refers to the underlying strict partition).  The proof is similar to that of Proposition 2.4; one must show that for $\nu$ as in Lemma 2.5, the specialization  $Q \mapsto \prod_{i=1}^r (1-t_{\nu_i+1})/(1+t_{\nu_i+1})$ sends $4\R_{k \, l}$ to $0$; this too is an elementary calculation. 

The proof of the formula (2) in the introduction, comparing the expansion of $\R_{\lambda}$ and that of $\Q_{\lambda(+)}$, where $\lambda(+) = (\lambda_1+1, \dots, \lambda_r+1)$, also follows from 
these specialization techniques, by verifying that both expansions vanish by the same specializations.  

As in type C, these polynomials are unchanged if $e(\lambda_i)$ is replaced by $e(\lambda_{i+1})$, 
whenever $\lambda_i = \lambda_{i+1} + 1$.  In particular, they are symmetric in the variables $t_j$ and $t_{j+1}$ if $j$ is not one of the integers $\lambda_i$ for which $\lambda_i > \lambda_{i+1} + 1$ or $j = r$ (i.e.~ $j$ is not a row number of a corner of the skew diagram of $\lambda$, allowing the last row of this diagram to have an empty corner).   

More generally, for $\lambda$ a partition of type $D$, we define a multi-Schur Pfaffian, 
denoted 
$\Pf_{\lambda}(\,c(1)|d(1),\ldots,c(r)|d(r)\,)$ 
as follows.
The elements $c(i)$ are polynomials of degree at most $\lambda_i$, with constant term $1$, and each $c(i)$ is required to divide $d(i)$ as well as all $c(j)$ for $j<i$.  Furthermore, they satisfy $d(i) \, d(j)^* = c(i)\, c(j)^*$ for all $i<j$.  The Pfaffian is defined from the same matrix used to define $\R_\lambda$, with $d(k)$ replacing $r(\lambda_k)$ and $c(k)$ replacing $e(\lambda_k)$, and the factors of $1/2$ and $1/4$ are removed.  

Note that, unlike the situations in types B and C, one must specify both the $d(i)$ and the $c(i)$, and not just the former, to form this Pfaffian.  When all the polynomials $c(i)$ have degree strictly less than $\lambda_i$, though, this reduces to $\Pf_\lambda( d(1),\ldots,d(r) )$.

A {\bf{triple}} $\bm{\tau} = (\bk, \bp, \bq)$ {\bf{of type D}} consists of 
three sequences of integers of the same length  $s$, with the $k$'s still required to be positive, but the $p$'s and $q$'s allowing zeros:
\[
\bk  =    0 < k_1 < \dots < k_s, \,\,\,\, \bp  =   p_1 \geq \dots \geq p_s \geq 0, \,\,\,\,
\bq  =  q_1 \geq \dots \geq  q_s \geq 0,
\]
satisfying the same conditions that  
\[
\tag{*} \,\,\,\,\, (p_i - p_{i+1}) + (q_i - q_{i+1}) \, > \, k_{i+1} - k_i 
\]
for $1 \leq i \leq s-1$.  We set $r = k_s$.  A triple $\bm{\tau}$ of type D determines a partition $\lambda = \lambda(\bm{\tau})$ of type D and length $r$, defined 
to  be the minimal strict partition with 
\[ 
\lambda_{k_i} = p_i + q_i.
\]
Note that this allows $\lambda_{r}$ to be $0$, when $p_s = q_s = 0$.

The signed permutation $w(\bm{\tau})$of a triple $\bm{\tau}$ is the $w$ of minimal length such 
that the number of $a > p_i$ with $w(a) = \bar{b}$ and $b > q_i$.  It can be constructed by a 
similar prescription as in type C, but using positions and barred values strictly greater rather than 
weakly greater.  
Alternatively, to a triple $\bm{\tau}$ of type D, one can construct 
a triple 
 $\bm{\tau(+)}$ of type C or B, 
by replacing each $p_i$ by $p_i+1$ and replacing each $q_i$ by $q_i+1$.  
Each triple of type C arises 
uniquely in this way.  As signed permutations, $w(\bm{\tau(+)}) = w(\bm{\tau})$.
The partitions are different, however: $\lambda(\bm{\tau(+)}) = \lambda(\bm{\tau})(+)$.  The 
signed permutations arising this way are the {\em vexillary} signed permutations of type D.  
One gets the same signed permutations as in types B and C, but note that, in addition to 
those in the usual Weyl group, consisting of signed permutations with an even number of 
barred values, we also get those in the coset consisting of those with an odd number.

As in type C, one can allow redundant triples, with the strict inequalities in (*) replace by weak inequalities; the fact that the polynomials don't change when one reduces a redundant triple to a triple is valid for these polynomials, for essentially the same reason (see Lemma 2.2).   As in type C, 
any triple can be extended to a redundant triple with $\bk = 1\;2\;\cdots\;r$.

The polynomial ${\bf{R}}_{\bm{\tau}}(x,y)$ in $\Gamma'[x,y]$ determined by a triple $\bm{\tau}$ of type D can be defined by substitution from $\R_{\lambda({\bm{\tau}})}(t)$, as in type C.  One substitutes (in any order) $x_1, \dots, x_{p_s}, y_1, \dots, y_{q_s}$ for $t_1, \dots, t_{\lambda_s}$, and then substitutes $x_{p_s+1}, \dots, x_{p_{s-1}}, y_{q_s+1}, \dots, y_{q_{s-1}}$ for $t_{\lambda_s+1}, \dots, t_{\lambda_{s-1}}$, continuing the substitution in groups (and using the symmetry of the polynomials $\R_{\lambda}(t)$). 
Equivalently, 
\begin{equation}\label{e.rtau}
{\bf{R}}_{\bm{\tau}}(x,y) = \frac{1}{2^r}\Pf_{\lambda(\bm{\tau})}(\, c(1)|d(1),\ldots,c(r)|d(r)\,),
\end{equation}
where $c(k_i) = \prod_{j=1}^{p_i}(1+x_j)\cdot \prod_{j=1}^{q_i}(1+y_j)$, $c(k) = c(k_i)$ for 
$k_i$ minimal with $k_i \geq k$, and $d(k) = Q\cdot c(k)$ for all $k$.

\smallskip

The geometric setting for degeneracy loci in type D is the following.  We have a vector bundle $V$ of rank $2n$, with a nondegenerate symmetric form, and flags $V \supset E_0 \supset E_1 \supset \cdots $ and  $V \supset F_0 \supset F_1 \supset \cdots$ of isotropic subbundles.  Here $\rk E_i = \rk F_i = n-i$.  
Set $E = E_0$ and $F = F_0$.    The specialization will take $Q$ to 
\[
c(V - E - F) = c(E^*)/c(F) = c(F^*)/c(E). 
\]
In type D, $P_i$ will map to $\frac{1}{2}c_i(V - E - F)$.  

For $i=1,\ldots,n$, set $x_i = c_1(E_{i-1}/E_i)$ and $y_i = c_1(E_{i-1}/E_i)$.  
Note that $d(\lambda_{k_i})$ maps to $c(V-E_{p_i}-F_{q_i})$.

Recall that in type D there are two families of maximal isotropic bundles, depending on the dimension of intersection of $E$ and $F$, which is constant modulo $2$ on connected components.  We assume for simplicity that $X$ is connected.  To choose a component, the convention we use, which is compatible with stability, is to require that the dimension of $E \cap F$ be even when 
the specified integer $r$ is even, and odd when $r$ is odd.  Note that this does not depend on the rank $n$ of $E$ and $F$, so it is compatible with stability.  (Note, however, that $E$ and $F$ will be in the same family in the conventional sense when $r$ and $n$ have the same parity, and in opposite families otherwise.)  

The locus $\Omega_{\triple}$ for a triple can be described by requiring that the dimension of the intersection $E_{p_i} \cap F_{q_i}$ be at least $k_i$ for all $i$, but now, in addition to taking 
care that the determinants do not cut out the reduced loci, one must also take this locus to be the closure of the reduced locus where all the intersections have the asserted dimension.  (See \cite{F3} and \cite{FP} for a discussion of this.)  

The double Schubert polynomials $\mathfrak{D}_w$, for any signed permutation $w$, are constructed by  
the same procedure as in types C and B, but with 
the following differences.  There are two top classes for a given $n$.   One comes from the triple 
\[
\bm{\tau} = (1 \; 2\, \ldots\, n\!-\!1, \,\, n\!-\!1 \,\ldots\, 2 \; 1, \,\, n\!-\!1 \, \ldots \,2 \; 1),
\]
 with $\lambda(\bm{\tau}) = (2n-2, \dots, 2)$, and $w(\bm{\tau}) = 1\; \bar{2} \ldots \bar{n}$.  The other comes from the triple 
\[
\bm{\tau} = (1 \; 2 \,\ldots\, n, \,\, n\!-\!1 \ldots 2 \; 1 \; 0, \,\, n\!-\!1 \,\ldots\, 2 \; 1 \; 0),
\]
with $\lambda(\bm{\tau}) = (2n-2, \dots, 2, 0)$, and $w(\bm{\tau}) = \bar{1}\; \bar{2} \; \cdots\; \bar{n}$.  
So one has the top polynomials $\mathfrak{D}_{w(\bm{\tau})}$ for these vexillary signed permutations.

The other polynomials are defined from these by sequences of difference operators, as before.  The $s_i$ and $\partial_i$ for $i > 0$ are the same as in types A, C, and B.  The new generator of the Weyl group is $s_{\hat{1}}$, which can be identified with $s_0s_1s_0$ in type C or B.  This acts on $\Gamma'[x,y]$ by taking $x_1$ to $-x_2$, $x_2$ to $-x_1$, and leaving the other $x$'s and all $y$'s unchanged.  One 
calculates that, setting  $v_j = \sum_{a+b=j} x_1^ax_2^b$,
\[
  s_{\hat{1}}(P_k) = P_k + 2(x_1+x_2)\left(\sum_{j=1}^{k-1} v_{j-1} P_{k-j} \right) + (x_1+x_2) v_{k-1} .
\]
The difference operator is $\partial_{\hat{1}}(f) = (f - s_{\hat{1}}(f))/(-x_1 - x_2)$, and one has 
\[ 
\partial_{\hat{1}}(P_k) = 2 \sum_{j=1}^{k-1} v_{j-1} P_{k-j} + v_{k-1}.
\]

\begin{thm}\label{t.mainD}
{\rm (1)} The class of $\Omega_{\bm{\tau}}$ is ${\bf{R}}_{\bm{\tau}}(x,y)$. 

\smallskip
\noindent
{\rm (2)} The polynomial ${\bf{R}}_{\bm{\tau}}(x,y)$ is equal to the double Schubert polynomial 
$\frak{D}_w(x,y)$, where $w = w(\bm{\tau})$.
\end{thm}

The proof of the theorem is similar in spirit to the proofs in types C and B, but we need to work with quadric bundles instead of projective bundles.  As before, we may assume $\bk=1\;2\;\cdots\;r$.  We resolve the locus $\Omega_\triple\subseteq X$ by a sequence of maps
\[
  Z_r \xrightarrow{\phi_r} Z_{r-1} \xrightarrow{\phi_{r-1}} \cdots \xrightarrow{\phi_2} Z_1 \xrightarrow{\phi_1} X,
\]
where each $Z_k$ is a locus in a quadric bundle over $Z_{k-1}$.  To start, inside the quadric bundle $\mathcal{Q}(V)$ over $X$, we have $Z_1 = \P(E_{p_1}) \cap \P(F_{q_1})$.  Letting $D_1$ be the tautological bundle on $\mathcal{Q}(V)$ restricted to $Z_1$, we have the quadric bundle $\mathcal{Q}(D_1^\perp/D_1)$ on $Z_1$.  Then $Z_2 = \P(E_{p_2}/D_1) \cap \P(E_{q_2}/D_1) \subset \mathcal{Q}(D_1^\perp/D_1)$, and $D_2/D_1$ is the tautological bundle of $\mathcal{Q}(D_1^\perp/D_1)$ restricted to $Z_2$, and so on.

As before, the composed map $\phi\colon Z_r \to X$ is birational onto $\Omega_\triple$, so to prove the theorem, we need to show that $\phi_*(1)$ is equal to the Pfaffian ${\bf{R}}_\triple(x,y)$.  This is done inductively, as in type C (and in \cite{K}), factoring $\phi_*$ as $(\phi_1)_*\cdots(\phi_r)_*$.

The key step in the calculation is this.  For $1\leq k \leq r$, let $h_k = -c_1(D_k/D_{k-1})$, and $H^{(k)} = c(D_{k-1}-D_{k-1}^*)$, so
\[
  c(D_{k-1}-D_{k-1}^*) = \prod_{i=1}^{k-1} \frac{1-h_i}{1+h_i} = \prod_{i=1}^{k-1} (1-2h_i+2h_i^2 -\cdots) = 1 + H^{(k)}_1 + H^{(k)}_2 + \cdots.
\]
For the induction, it is necessary to compute the value $(\phi_k)_*$ on an arbitrary monomial in $h_1,\ldots,h_k$.  It suffices to do this for powers of $h_k$, using the type D analogue of Lemma~\ref{l.typeCstep}:

\begin{lem*}
For $1\leq k\leq r$, and for $m\geq 0$, we have
\begin{align*}
 (\phi_k)_*(h_k^m) &= \frac{1}{2}\left( \sum_{j\geq 0} (-1)^j c_{p_k+q_k+m-j}(V-E_{p_k}-F_{q_k})\, H^{(k)}_j \right) \\
  & \qquad + \frac{1}{2} (-1)^{r-k}\delta_{m,0}\, c_{p_k+q_k}(E/E_{p_k}+F/F_{q_k}),
\end{align*}
where $\delta_{m,0}$ is the Kronecker delta.
\end{lem*}

The case $k=1$ is a direct calculation on the quadric bundle $\mathcal{Q}(V)$, using the presentation of Edidin and Graham \cite{EG}.  (Since $H^{(1)}=1$, the formulas are much simpler in this case.)  For $k>1$, it is the same, but one replaces $V$ with $D_{k-1}^\perp/D_{k-1}$, etc.

The formula in part (1) of Theorem~\ref{t.mainD} is a consequence of the Lemma together with the algebraic identities proved in the Appendix.  
Indeed, the Pfaffian formula for ${\bf{R}}_\triple(x,y)$ is precisely what appears in Proposition~\ref{p.gysin} after substituting $c(k)$ for $g(k)$.  
(The ring $A$ that appears there is the cohomology ring of $X$, and the ring $A_k$, or more precisely its finite-dimensional quotient, is the part of the cohomology ring of $Z_k$ that is generated over $A$ by the Chern classes $h_1,\ldots,h_k$.)

\smallskip

The diagonal trick to reduce the theorem to the case of maximal isotropics is almost 
the same as in type C, using $V' = V \oplus V$, with the symmetric bilinear form 
$\langle \, (u,v), (u',v') \, \rangle = \langle u, u' \rangle - \langle v , v' \rangle$.  
Again the diagonal $\Delta_V$ and $E \oplus F$ are isotropic, and the argument concludes 
as in type C.

\smallskip

As the type C theorem applies to symmetric maps of vector bundles, 
the type D theorem applies to skew-symmetric maps.  For this, let 
$E$ be a vector bundle of rank $n$, with a flag of 
subbundles $E = E_0 \supset E_1 \supset \dots $, with $E_p$ of 
corank $p$ in $E$.  Let $F$ be a subbundle of $E$, of corank $q$, and 
let $\phi \colon F \to E^*$ be a skew-symmetric map of vector bundles, 
which means that the induced composition $F \to E^* \to F^*$ is equal 
to minus its dual. (In characteristic $2$ one should also require that $\phi(f)(f) = 0$ 
for all $f$ in $F$.)
Given sequences of positive integers $k_1 < \dots < k_s$ 
and nonnegative integers $p_1 > \dots > p_s$, with  
$p_i - p_{i+1} > k_{i+1} - k_i$, we 
have the locus where the kernel $\Ker(\phi)$ meets $E_{p_i}$ in 
dimension at least $k_i$.  We assume that $r = k_s$ is congruent to $n$ 
modulo $2$; when $q = 0$, this can always be achieved by adding or 
removing a $0$ at the end of the string of $p_i$'s.
Define $\lambda$ to be the minimal strict 
partition (of type D) with $\lambda_{k_i} = p_i + q$.  
Let $c(k_i) = c(E/F)\cdot c(E/E_{p_i})$, 
set $c(k) = c(k_i)$ for $k_i \geq k$ minimal in general, and let $d(k) = c(E^*-E)\cdot c(k)$. 

\begin{cor*}
The class of the locus where $\Ker(\phi)$ meets $E_{p_i}$ in 
dimension at least $k_i$ is given by 
$\frac{1}{2^r}\Pf_\lambda(\, c(1)|d(1), \ldots, c(1)|d(r)\,)$.  
\end{cor*}

The proof is by setting $V = E \oplus E^*$, with its standard symmetric 
form: $\langle \, (e,\alpha), (e',\alpha') \, \rangle = \alpha(e') + \alpha'(e) $.  
Then $E = E \oplus 0$ is an isotropic subbundle of $V$, as is the graph $F_q$ of $\phi$.  Note that if $F_0$ is a maximal isotropic subbundle containing $F_q$, then the dimension of $E_0 \cap F_0$ has the same parity as $n$, and hence $r$.   
The formula is a special case of the main theorem, taking $q_i = q$ for all $i$. 

Again one has the special cases where $s = 1$ and where $q = 0$, which recover the 
formulas of \cite{K}.


\section*{Appendix on Pfaffian Algebra}

\setcounter{equation}{0}
\renewcommand{\theequation}{A.\arabic{equation}}
\setcounter{theorem}{0}
\renewcommand{\thetheorem}{A.\arabic{theorem}}

The goal of this appendix is to prove a formula for the composition of certain homomorphisms of power series rings, stated in Proposition~\ref{p.gysin} below.  The setting is purely algebraic, and the answer is a Pfaffian.

More details on Pfaffians can be found in \cite[Appendix D]{FP} or \cite{Knuth}; here we recall only the facts we need.  
For a skew-symmetric $r\times r$ matrix $M$, let $M_k$ be the submatrix with the $k$th row and column removed, and let $M_{k,l}$ be the submatrix with the $k$th and $l$th rows and columns removed.  When $r$ is even, the Pfaffian $\mathrm{Pf}(M)$ satisfies a Laplace-type expansion ``along the $l$th column,''
\[
  \mathrm{Pf}(M) = \sum_{k=1}^{l-1} (-1)^{k+l-1} m_{k,l}\, \mathrm{Pf}(M_{k,l}) + \sum_{k=l+1}^{r} (-1)^{k+l} m_{k,l}\, \mathrm{Pf}(M_{k,l}),
\]
which defines the Pfaffian inductively.  (Note that the expansions along the first and $r$th columns have slightly simpler formulas.)  For odd $r$, given a definition for $m_k$, we use the convention
\begin{equation}\label{e.oddpfaff}
  \mathrm{Pf}(M) = \sum_{k=1}^r (-1)^{k-1} m_k\, \mathrm{Pf}(M_k).
\end{equation}
This is equivalent to augmenting the matrix $M$ with a $0$th column and $0$th row, setting $m_{0,k}=-m_{0,k}=m_k$, and taking the Pfaffian of this augmented matrix.  The notation is slightly abusive, but the matrices we consider will always come with definitions for both $m_{k,l}$ and $m_k$.

The first general formula we prove is an identity of operators on rings of Laurent series.  In order to state it precisely, we first describe the ingredients.  Fix $r>0$, and a partition $\lambda$ of type D and length $r$.  Also fix a base ring $A$ (assumed to be commutative, with unit).  Consider the ring of formal Laurent series in variables $h_1,\ldots,h_r$ with coefficients in $A$; let $B$ be the result of adjoining polynomials in variables $u_1,\ldots,u_r$ to this ring.  For any subset $K\subseteq [r] := \{1,\ldots,r\}$, let $C_K \subseteq B$ be the subring where for $i\in K$, the variables $h_i$ appear with nonnegative exponents.  Finally, for $1\leq k\leq r$, let $B_k = C_{[k-1]}$ be the subring involving only nonnegative powers of the variables $h_1,\ldots,h_{k-1}$.  Note that $B=B_1$.

For any subsets $J\subseteq K\subseteq [r]$, we have an $A[u]$-algebra endomorphism $\zeta_J$ of $C_K$, defined by $\zeta_J(h_i)=0$ if $i\in J$, and $\zeta_J(h_i)=h_i$ if $i\not\in J$.  Since $\zeta_J$ is defined compatibly for $C_{K'} \subseteq C_K$ whenever $J\subseteq K\subseteq K'$, we can safely suppress its dependence on $K$ in the notation.  When $J_1$ and $J_2$ are disjoint subsets of $K$, we have $\zeta_{J_1}\zeta_{J_2}=\zeta_{J_2}\zeta_{J_1}=\zeta_{J_1\sqcup J_2}$.  Finally, we often write $\zeta_i$ for $\zeta_{\{i\}}$.

The function
\[
  f[i,j] = \frac{1-h_i/h_j}{1+h_i/h_j},
\]
considered an element of $C_{[r]\setminus \{i,j\} }$, is skew-symmetric in $i$ and $j$.  We set $f[k] = 1$.  For an index set $I\subseteq [r]$, let $M$ be the matrix with rows and columns indexed by $I$, with $m_{i,j}=f[i,j]$ and $m_k=f[k]=1$, and let $f[I]$ be the corresponding Pfaffian, with the interpretation of \eqref{e.oddpfaff} when the cardinality of $I$ is odd. 
As shown in \cite{K} Appendix C, the function $f$ satisfies the identity
\begin{equation}\label{e.fident}
  f[I] = \prod_{ \{i<j\} \in \binom{I}{2} } f[i,j].
\end{equation}

For our fixed partition $\lambda$, we define operators $\tilde{f}[i,j]\colon C_{\{i,j\}} \to B$ and $\tilde{f}[k]\colon C_{\{k\}} \to B$ by
\begin{align*}
  \tilde{f}[i,j]    &= h_i^{\lambda_i} h_j^{\lambda_j} f[i,j] + h_i^{\lambda_i} u_j^{\lambda_j} \zeta_j - u_i^{\lambda_i} h_j^{\lambda_j} \zeta_i - u_i^{\lambda_i} u_j^{\lambda_j} \zeta[i,j] \\
 \intertext{and}
 \tilde{f}[k] &= h_k^{\lambda_k} + u_k^{\lambda_k}\zeta_k,
\end{align*}
where $\zeta[i,j]$ is defined to be $\zeta_{\{i,j\}}$ when $i<j$ and $-\zeta_{\{i,j\}}$ when $i>j$.  This makes the operators $\tilde{f}[i,j]$ skew-symmetric in $i$ and $j$.  When $i<j$, we have
\begin{align}\label{e.fentry}
  \tilde{f}[i,j] &= (h_i^{\lambda_i} - u_i^{\lambda_i} \zeta_i )(h_j^{\lambda_j} + u_j^{\lambda_j}\zeta_j) + 2\sum_{k>0} (-1)^k h_i^{\lambda_i+k} h_j^{\lambda_j-k} .
\end{align}
When $i,j,k,l$ are distinct, the product $\tilde{f}[i,j]\,\tilde{f}[k,l] = \tilde{f}[k,l]\,\tilde{f}[i,j]$ is an operator from $C_{\{i,j,k,l\}}$ to $B$.  In particular, when $|K|$ is even, for any decomposition $K$ into disjoint pairs $\{i_p<j_p\}$, the product $\tilde{f}[i_1,j_1]\,\tilde{f}[i_2,j_2]\cdots$ is an operator $C_K \to B$.  Similarly, when $|K|$ is odd, a decomposition into such pairs together with a remainder $\{k\}$ yields an operator $\tilde{f}[k]\,\tilde{f}[i_1,j_1]\,\tilde{f}[i_2,j_2]\cdots$ from $C_K$ to $B$.  This implies that the Pfaffian of a matrix with rows and columns indexed by $K$, and with entries $\tilde{f}[i,j]$, is well-defined as an operator $C_K \to B$.

For a subset of indices $I=\{i_1,\ldots,i_p\}\subseteq [r]$, write $h^I = h_{i_1}^{\lambda_{i_1}} \cdots h_{i_p}^{\lambda_{i_p}}$, and similarly define $u^I = u_{i_1}^{\lambda_{i_1}}\cdots u_{i_p}^{\lambda_{i_p}}$.

Finally, we define notation for signs.  Given a subset $K = \{k_1<\cdots<k_s\} \subseteq [r]$, the parity of an element $k=k_i$ is defined to be the parity of $i$.  For $k\in K$, set
\begin{align*}
  \epsilon(k,K) &= \begin{cases} 1 &\text{ if }k\text{ is an even element of }K; \\ -1 &\text{ if }k\text{ is an odd element of }K, \end{cases} \\
\intertext{and, for a subset $J\subseteq K$, define a sign}
  \mathrm{sgn}(K,J) &= (-1)^{|J|\cdot|K|}(-1)^{\#(\text{odd elements of }J)}.
\end{align*}

We will need some identities satisfied by these signs.
\begin{lemma}
Set $s=|K|$, and consider a subset $J=\{j_1<\cdots<j_t\} \subseteq K$.
\begin{enumerate}
\item We have $\mathrm{sgn}(K,J) = (-1)^{\lfloor s/2\rfloor}\mathrm{sgn}(K,K\setminus J)$.

\medskip

\item If $s=|K|$ is odd, then we have
\[
  -\epsilon(j_p,K)\, \mathrm{sgn}(K\setminus j_p, J \setminus j_p) = (-1)^{p-1}\mathrm{sgn}(K,J),
\]
and therefore
\[
  \sum_{j\in J} -\epsilon(j,K)\, \mathrm{sgn}(K\setminus j, J\setminus j) = \begin{cases} \mathrm{sgn}(K,J) & \text{ if }|J|\text{ is odd;} \\ 0 & \text{ if }|J|\text{ is even.} \end{cases}
\]
\end{enumerate}
\end{lemma}

The key algebraic identity is this:
\begin{proposition}\label{p.main}
Fix $K\subseteq [r]$.  For $i,j,k\in K$, set $m_{i,j} = \tilde{f}[i,j]$ and $m_k=\tilde{f}[k]$.  Let $M$ be the corresponding skew-symmetric matrix with rows and columns indexed by $K$.  We have an identity of operators $C_K \to B$
\begin{align}\label{e.main}
   \mathrm{Pf}( M ) &= \sum_{ I \sqcup J = K } \mathrm{sgn}(K,J)\, h^I\, u^{J}\,  f[I] \,\zeta_{J},
\end{align}
using \eqref{e.oddpfaff} to interpret the left-hand side when $|K|$ is odd.
\end{proposition}

\begin{proof}
The proof is by induction on $s=|K|$.  First suppose $s$ is odd.  The LHS is
\[
  \sum_{p=1}^s (-1)^{p-1} (h_{k_p}^{\lambda_{k_p}}+u_{k_p}^{\lambda_{k_p}} \zeta_{k_p}) \mathrm{Pf}(M_{k_p}),
\]
and by induction, we know the formula for $\mathrm{Pf}(M_{k_p})$.  For $I\sqcup J = K$, the coefficient of $h^I u^J \zeta_J$ in this expression is
\begin{equation}\label{e.twosums}
  \sum_{i\in I} -\epsilon(i,K)\,  \mathrm{sgn}( K\setminus i, J ) f[I\setminus i] + \sum_{j\in J} -\epsilon(j,K)\, \mathrm{sgn}( K\setminus j, J\setminus j ) f[I].
\end{equation}
We need to show that this is equal to $\mathrm{sgn}(K,J)\, f[I]$, the corresponding coefficient on the RHS.

When $|I|$ is odd, so $|J|$ is even, the second sum in \eqref{e.twosums} is zero.  To see that first sum is equal to $\mathrm{sgn}(K,J)\,f[I]$, by the previous lemma we have
$\mathrm{sgn}(K\setminus i , J) = (-1)^{(s-1)/2} \mathrm{sgn}(K\setminus i, I\setminus i)$, so
\[
 (-1)^{(s-1)/2} \mathrm{sgn}(K,I) \sum_{p=1}^{|I|} (-1)^{p-1} f[I\setminus i_p] = \mathrm{sgn}(K,J) f[I].
\]

When $|I|$ is even, so $|J|$ is odd, the second of these sums is equal to $\mathrm{sgn}(K,J)\,f[I]$.  We claim that the first sum is zero in this case.  To see this, expand $f[I\setminus i]$ and write the sum as
\[
  \pm \sum_{p=1}^{|I|} (-1)^{p-1} \left( \sum_{q=1}^{p-1}(-1)^{q-1}f[I\setminus\{i_p,i_q\}] + \sum_{q=p+1}^{|I|} (-1)^q f[I\setminus\{i_p,i_q\}] \right) ,
\]
where the sign in front is $(-1)^{(s-1)/2}\mathrm{sgn}(K,I)$.  For each fixed $p$, the summand is $f[I]$, by the Laplace expansion.  Since there are $|I|$ summands of alternating sign, the sum is zero.

When $s$ is even, use the Laplace expansion to write the LHS of \eqref{e.main} as a sum of size $s-2$ Pfaffians.  Then use induction and compare coefficients as in the odd case.
\end{proof}

\subsection*{An algebraic pushforward}

For each $k$, set $A_k = A[[h_1,\ldots,h_{k}]]$.  We will consider algebraic ``pushforward'' homomorphisms $(\phi_k)_*\colon A_k \to A_{k-1}$.  The main formula computes the composition of these in terms of a certain multi-Schur Pfaffian.

We describe the Pfaffian first.  Fix a partition $\lambda$ of type D and length $r$.  
Assume that the base ring $A$ is graded, and for each $1\leq k\leq r$, suppose we are given series of (commuting) elements $d(k) = 1+d(k)_1+d(k)_2+\cdots$ and $g(k)=1+g(k)_1+g(k)_2+\cdots$.  The $g(k)$ are required to be finite, with $g(k)_m=0$ for $m>\lambda_k$, and each $c(i)$ divides $d(i)$, as well as $g(j)$ for $j>i$.  Furthermore, the $d$'s and $g$'s are required to satisfy the relations $d(k) d(l)^* = g(k) g(l)^*$, where $d(l)^* = 1 - d(l)_1 + d(l)_2 - \cdots$.

The multi-Schur Pfaffian we need is defined along the lines of Section 4 in the text.  Form an $r\times r$ matrix $M$ with
\[
  m_{k,l} = (d(k)_{\lambda_k}-g(k)_{\lambda_k})( d(l)_{\lambda_l}+g(l)_{\lambda_l} )  + 2\sum_{j=1}^{\lambda_l} (-1)^j d(k)_{\lambda_k+j}\, d(l)_{\lambda_l-j},
\]
and set $m_k = d(k)_{\lambda_k} + g(k)_{\lambda_k}$.  The relations on $d$ and $g$ guarantee that $M$ is skew-symmetric, and we write $\Pf_\lambda(\, g(1)|d(1),\ldots, g(r)|d(r) \,)$, for its Pfaffian, using \eqref{e.oddpfaff} to interpret this when $r$ is odd.

The homomorphisms $(\phi_k)_*$ are $A_{k-1}$-linear, and are defined by\footnote{If the elements $d(k)_m$  are zero for all sufficiently large $m$---say, for $m\geq N$---then the same formulas define maps $A_k/I\, A_k \to A_{k-1}/I\, A_{k-1}$, 
where $I$ is the ideal generated by all monomials in $h$ of degree $N$.  We only point this out in case one prefers to work with finite-dimensional extensions of $A$, such as the those that appear as cohomology rings in the geometric setting.} 
\[
  (\phi_k)_*(h_k^m) = \frac{1}{2}\left( \sum_{j\geq 0} (-1)^j d(k)_{\lambda_k+m-j}\, H^{(k)}_j \right) + \frac{1}{2} (-1)^{r-k}\delta_{m,0}\, g(k)_{\lambda_k},
\]
where $H^{(k)} = \prod_{i=1}^{k-1}\frac{1-h_i}{1+h_i}$.

\begin{proposition}\label{p.gysin}
Given a partition $\lambda$ of type D and homomorphisms $(\phi_k)_*$ as above, we have
\begin{equation}\label{e.formula2}
  (\phi_1)_*\cdots(\phi_r)_*(1) = \frac{1}{2^r}\Pf_\lambda( \,g(1)|d(1),\ldots,g(r)|d(r) \,).
\end{equation}
\end{proposition}

To prove the proposition, recall the rings $B_r \subset \cdots \subset B_1=B$ (of Laurent series in $h$ and polynomials in $u$).  Define an $A_{k-1}$-linear homomorphism\footnote{As with the homomorphisms $(\phi_k)_*$, when the elements $d(k)_m$ are zero for all sufficiently large $m$, the same formula defines a map $\bar\rho_k\colon B_k/I\, B_k \to A_{k-1}/I\, A_{k-1}$.} $\rho_k\colon B_k \to A_{k-1}$
by
\[
  \rho_k(h_k^{m_k}\cdots h_r^{m_r}\,u_1^{n_1}\cdots u_r^{n_r}) = d(k)_{m_k}\cdots d(r)_{m_r} g(1)_{n_1} \cdots g(r)_{n_r}.
\]
Using $\rho$ and $\zeta$, the homomorphism $(\phi_k)_*$ can be rewritten as
\begin{align*}
  (\phi_k)_*(h_k^m) &= \frac{1}{2}\rho_k\left( \sum_{j\geq 0} (-1)^j h_k^{\lambda_k+m-j}\, H^{(k)}_j + (-1)^{r-k} u_k^{\lambda_k} \zeta_k(h_k^m) \right) \\
  &= \frac{1}{2}\rho_k\left( h_k^{\lambda_k} \prod_{i=1}^{k-1} \frac{1-h_i/h_k}{1+h_i/h_k} + (-1)^{r-k} u_k^{\lambda_k} \zeta_k \right)(h_k^m).
\end{align*}
That is, after applying $\rho_k$, the homomorphism $(\phi_k)_*$ becomes multiplication by the operator $\Phi_k := (1/2)\left( h_k^{\lambda_k} \prod_{i=1}^{k-1} \frac{1-h_i/h_k}{1+h_i/h_k} + (-1)^{r-k} u_k^{\lambda_k} \zeta_k \right)$.

Using linearity, the composition $(\phi_{k-1})_*(\phi_k)_*$ acts by sending $h_{k-1}^m h_k^n$ to
\begin{align*}
 \rho_{k-1}\left( \Phi_{k-1}\, h_{k-1}^m \rho_k( \Phi_k\, h_k^n )  \right)  = \rho_{k-1}\left( \Phi_{k-1}\Phi_k h_{k-1}^m h_k^n \right) ,
\end{align*}
and similarly, the composition $(\phi_1)_*\cdots(\phi_r)_*$ multiplies $h_1^{m_1}\cdots h_r^{m_r}$ by the product of these operators, $\rho_1\,\Phi_1\cdots\Phi_r$.  This product expands as
\begin{align}\label{e.expanded}
\frac{1}{2^r}\rho_1\left( \sum_{ I \sqcup J = [r] } \mathrm{sgn}([r],J)\, h^I\, u^{J} \prod_{ \{i<j\} \in \binom{I}{2} } \frac{1-h_i/h_j}{1+h_i/h_j}\, \zeta_{J} \right),
\end{align}
using $\mathrm{sgn}([r],J) = (-1)^{\sum \{r-j\,|\, j\in J\}}$.  The operator inside the parentheses in this expression is exactly the RHS of the identity \eqref{e.main} of Proposition~\ref{p.main}, using the identity \eqref{e.fident} to substitute $f[I]$.

Recall the matrix $M$ used to define the right-hand side of \eqref{e.formula2}.  By construction, we have $m_k=\rho_1(\tilde{f}[k])$, and by \eqref{e.fentry}, we have $m_{i,j}=\rho_1(\tilde{f}[i,j])$.  Linearity of $\rho_1$ combined with Proposition~\ref{p.main} (using the case $K=[r]$) proves the identity \eqref{e.formula2}. \qed

\section*{Acknowledgements}
The first and second authors were partially supported by NSF grants DMS-0902967 and DMS-0502170.


\end{document}